\documentclass[12pt,a4paper, oneside, reqno]{amsart}
\usepackage{amsmath,amssymb,amsthm,graphicx,color,bbm,array}
\usepackage[left=1in,right=1in,top=1in,bottom=1in]{geometry}
\usepackage{cite}
\usepackage[foot]{amsaddr}
\usepackage[british]{babel}
\usepackage{hyperref} 
\usepackage{enumitem}
\usepackage{tikz,pgfplots}
\usetikzlibrary{decorations.markings}
\usepackage{caption}
\usepackage{subcaption}
\pgfplotsset{ticks=none}
\pgfplotsset{compat=1.18}

\newtheorem{theorem}{Theorem}[section]
\newtheorem{corollary}[theorem]{Corollary}
\newtheorem{proposition}[theorem]{Proposition}
\newtheorem{lemma}[theorem]{Lemma}

\numberwithin{equation}{section}

\theoremstyle{definition}
\newtheorem{definition}[theorem]{Definition}
\newtheorem{examplex}[theorem]{Example}
 \newtheorem*{problem}{Open problem}

\newenvironment{example}
  {\pushQED{\qed}\examplex}
  {\popQED\endexamplex}

\theoremstyle{remark}
\newtheorem{remark}[theorem]{Remark}
\newtheorem{remarks}[theorem]{Remarks}
\newtheorem*{remark*}{Remark}

 \allowdisplaybreaks

\newcommand{\R}{{\mathbb R}}

\newcommand{\Z}{{\mathbb Z}}
\newcommand{\N}{{\mathbb N}}
\newcommand{\bbO}{{\mathbb O}}

\newcommand{\ZP}{{\mathbb Z}_+}
\newcommand{\RP}{{\mathbb R}_+}

\newcommand{\Sp}[1]{{\mathbb S}^{#1}}

\DeclareMathOperator{\Exp}{\mathbb{E}}
\let\Pr\relax
\DeclareMathOperator{\Pr}{\mathbb{P}}

\DeclareMathOperator{\Var}{\mathbb{V}ar}

\DeclareMathOperator{\arcsinh}{arcsinh}

\DeclareMathOperator{\sd}{\triangle}

\DeclareMathOperator*{\capa}{\mathrm{cap}}

\newcommand{\tra}{{\scalebox{0.6}{$\top$}}}
\newcommand{\per}{{\mkern -1mu \scalebox{0.5}{$\perp$}}}

\newcommand{\eps}{\varepsilon}

\newcommand{\ud}{{\mathrm d}}

\newcommand{\cA}{{\mathcal A}}
\newcommand{\cB}{{\mathcal B}}
\newcommand{\cBp}{{\cB}_+}
\newcommand{\cC}{{\mathcal C}}

\newcommand{\cE}{{\mathcal E}}
\newcommand{\cF}{{\mathcal F}}
\newcommand{\cG}{{\mathcal G}}
\newcommand{\cH}{{\mathcal H}}

\newcommand{\cL}{{\mathcal L}}

\newcommand{\cN}{{\mathcal N}}

\newcommand{\fKdo}{{\mathfrak{K}_d^0}}
\newcommand{\fCdo}{{\mathfrak{C}_d^0}}
\newcommand{\Sigmap}{\Sigma_{\mu^\per}}
\newcommand{\fg}{{\mathfrak{g}}}
\newcommand{\seg}{{\mathfrak{s}}}
\newcommand{\tlambda}{{\widetilde \lambda}}

\newcommand{\cCd}{\cC_d}
\newcommand{\cAx}{\cA^{x}}
\newcommand{\cCdo}{\cC_d^0}

\newcommand{\hull}{\mathop \mathrm{hull}}
\newcommand{\cl}{\mathop \mathrm{cl}}

\newcommand{\diam}{\mathop \mathrm{diam}}

\newcommand{\as}{\ \text{a.s.}}

\newcommand{\obar}[1]{\mkern 1.5mu\overline{\mkern-1.5mu#1\mkern-1.5mu}\mkern 1.5mu}
\newcommand{\ubar}[1]{\underline{#1\mkern-4mu}\mkern4mu }
\newcommand{\of}{{\obar f}}
\newcommand{\uf}{{\ubar f}}

\newcommand{\og}{{\obar g}}
\newcommand{\ug}{{\ubar g}}
\newcommand{\hf}{{\hat f}}
\newcommand{\fs}{f^\mathrm{s}}
\newcommand{\fopt}{f^\star}

\newcommand{\0}{{\mathbf{0}}}

\makeatletter
\def\namedlabel#1#2{\begingroup  
    (#2)%
    \def\@currentlabel{#2}%
    \phantomsection\label{#1}\endgroup
}
\makeatother

\newlist{myenumi}{enumerate}{10}
\setlist[myenumi]{leftmargin=0pt, labelindent=\parindent, listparindent=\parindent, labelwidth=0pt, itemindent=!, itemsep=1pt, parsep=4pt}

\newlist{thmenumi}{enumerate}{10}
\setlist[thmenumi]{leftmargin=0pt, labelindent=\parindent, listparindent=\parindent, labelwidth=0pt, itemindent=!}

\title[Iterated-logarithm laws for convex hulls of random walks with drift]{Iterated-logarithm laws for convex hulls \\ of random walks with drift}

\author[W.\ Cygan]{Wojciech Cygan$^{1,2}$}
\address{$^{1}$University of Wroc\l{}aw,
		Faculty of Mathematics and Computer Science\\
		Institute of Mathematics,
		pl.\ Grunwaldzki 2/4, 50--384 Wroc\l{}aw, Poland}
\address{$^{2}$Technische Universit\"{a}t Dresden,
		Faculty of Mathematics\\
		Institute of Mathematical Stochastics,
		Zellescher Weg 25, 01069 Dresden, Germany}
\email{wojciech.cygan@uwr.edu.pl}

\author[N.\ Sandri\'{c}]{Nikola Sandri\'{c}$^{3}$}
\address{$^{3}$Department of Mathematics\\University of Zagreb\\ Zagreb\\Croatia}
\email{nsandric@math.hr}

\author[S.\ \v{S}ebek]{Stjepan\ \v{S}ebek$^{4}$}
\address{$^{4}$Department of Applied Mathematics\\
	Faculty of Electrical Engineering and Computing\\
	University of Zagreb\\ 
 Zagreb\\ 
	Croatia}
\email{stjepan.sebek@fer.hr}

\author[A.\ Wade]{Andrew Wade$^{5}$}
\address{$^{5}$Department of Mathematical Sciences\\
	Durham University\\
	Durham\\ 
	UK}
\email{andrew.wade@durham.ac.uk}

\subjclass[2010]{
60G50\! (Primary),
60D05, 60F15, 60J65, 52A22\! (Secondary)}
\keywords{Random walk; convex hull; intrinsic volumes; Strassen's theorem; law of the iterated logarithm; zero--one law; shape theorem.}

\begin{document}

%

\begin{abstract}
 We establish laws of the iterated logarithm for intrinsic volumes of the convex hull
 of many-step, multidimensional random walks 
 whose increments have two moments and a non-zero drift.
 Analogous results in the case of zero drift, where the scaling is different, were obtained by Khoshnevisan. Our starting point is a version of Strassen's functional law of the iterated logarithm for random walks with drift. For the special case of the area of a planar random walk with drift, we compute explicitly the constant in the iterated-logarithm law by solving an isoperimetric problem reminiscent of the classical Dido problem. For general intrinsic volumes and dimensions, our proof exploits a novel zero--one law for functionals of convex hulls of walks with drift, of some independent interest. As another application of our approach, we obtain iterated-logarithm laws for intrinsic volumes of the convex hull of the centre of mass (running average) process associated to the random walk.
  \end{abstract}

\maketitle

%
%
%

\section{Introduction and main results}
\label{sec:intro}
 
Several fundamental aspects of the geometry of a stochastic process in Euclidean space 
are captured by its associated process of convex hulls, and so analysis of convex hulls of random processes
may be demanded by applications
of stochastic processes  in which geometry is important. For this reason, convex hulls of random walks and
diffusions, for example, have been studied motivated by models of animal movement in ecology, or by algorithms for set estimation.
We refer to~\cite{majumdar} for a survey of the state of the field around~2010;
milestones in the earlier work include~\cite{levy,sw,ss,chm,khoshnevisan}.

In the last decade or so, activity has increased significantly on several fronts; among many papers, we mention~\cite{av,css,bgm,klm,wx1,wx2,mcr,mcrw,kvz2,lmw,lhw}.
Several works consider the large-time asymptotics of random processes derived from geometrical functionals of the convex hull (such as volume, diameter, perimeter, and so on),
with results on expectation and variance asymptotics, laws of large numbers, distributional limits, and large deviations, for example.
In the present work, we consider \emph{iterated-logarithm} asymptotics, i.e., almost-sure quantification of the $\limsup$ growth rate
of quantities like the volume of the convex hull. Prior work here includes the deep 
contributions of L\'evy~\cite{levy} and Khoshnevisan~\cite{khoshnevisan}, but previous work considered only
the case where the walk has \emph{zero drift}.
The case with non-zero drift is, as is to be expected, quite different, and  that is our focus. We expect our approach could be adapted to the $\liminf$ growth rate, but we leave that for future work.

On a probability space $(\Omega, \cF, \Pr)$, let $Z, Z_1, Z_2, \ldots$ be a family of~i.i.d.~random variables in $\R^d$, $d \in \N := \{1,2,3,\ldots\}$, 
and let $S_n := \sum_{i=1}^n Z_i$ describe the associated multidimensional random walk in $\R^d$, started from $S_0 := \0$, the origin.
Let $\cH_n := \hull \{ S_0, S_1, \ldots, S_n \}$, where $\hull A$ denotes the convex hull of $A \subseteq \R^d$ (the smallest convex subset of $\R^d$ which contains~$A$).
Write $\| \, \cdot \, \|$ for the $d$-dimensional Euclidean norm. When vectors in $\R^d$ appear in formulas, they are to be interpreted as column vectors,
although, for typographical convenience, we sometimes write them as row vectors. 
Denote the unit sphere in $\R^d$ by~$\Sp{d-1} := \{ x \in \R^d : \| x \| = 1 \}$.
For $x \in \R^d \setminus \{ \0\}$ we set $\hat x := x / \| x\| \in \Sp{d-1}$;
we define $\hat \0 := \0$.
We will typically assume that the increments of the random walk
have finite second moments, and we will use the following associated notation; we write $\Exp$ for expectation under $\Pr$.

\begin{description}
\item
[\namedlabel{ass:moments}{M}]
Suppose that $\Exp ( \| Z \|^2 ) < \infty$, and denote the mean increment vector by
$\mu := \Exp Z$ (the drift) and the increment covariance matrix by
$\Sigma := \Exp ( ( Z -\mu) (Z - \mu)^\tra )$.
\end{description}

The goal of this paper is to establish laws of the iterated logarithm (LILs) for geometric functionals of $\cH_n$,
particularly in the case where $\mu \neq \0$. When $\mu = \0$, an elegant
result was provided by Khoshnevisan~\cite{khoshnevisan}. For example, when $d=2$ and~\eqref{ass:moments} holds with $\mu = \0$
and $\Sigma = I$ (identity),  Khoshnevisan established
that the area $A(\cH_n)$ satisfies
\begin{equation}
\label{eq:khoshnevisan-area-lil}
 \limsup_{n \to \infty}  \frac{A ( \cH_n )}{   n \log \log n } =  \frac{1}{\pi} , \as  ;
\end{equation}
 the analogue of this result for Brownian motion
had already been obtained by L\'evy~\cite{levy} (see Example~\ref{ex:volume-zero-drift} below for a result for volume of $\cH_n$ in general dimensions).
In fact, Khoshnevisan established~\eqref{eq:khoshnevisan-area-lil} under an inessential additional hypothesis, that coordinates of $Z$ are independent~\cite[p.~318]{khoshnevisan}, 
which we remove; see Theorem~\ref{thm:khoshnevisan} below.
If, still, $d=2$ and $\Sigma =I$, but we have $\mu \neq \0$,
we are in a new setting, and a different scaling is needed. A special case of our results (in Theorem~\ref{thm:constants} below) shows that, now
\begin{equation}
\label{eq:drift-area-lil}
 \limsup_{n \to \infty}  \frac{A ( \cH_n )}{   n^{3/2} \sqrt{ \log \log n }} =  \frac{ \| \mu \|}{\sqrt{6}} , \as  
\end{equation}

To allow us to present one further example, we define 
\begin{equation}
\label{eq:com-def}
G_0 := \0, \text{ and } G_n := \frac{1}{n} \sum_{i=1}^n S_i, \text{ for } n \in \N.
\end{equation}
The process $G_n$, $n\in \ZP$, is the \emph{centre of mass} process associated with the random walk.
Under assumption~\eqref{ass:moments}, $G_n$ satisfies law of large numbers and central limit theorem
asymptotics of the same order as $S_n$, but, locally, $G_n$ moves much more slowly,
which leads, for example, to the interesting fact that $G_n$ is compact-set transient when $\mu = \0$ and $d=2$; 
see~\cite{grill,lw,lmw} for these and other properties. 
Here we consider its convex hull, defined by $\cG_n := \hull \{ G_0, G_1, \ldots, G_n \}$.
By convexity, $G_n \in \cH_n$ and so $\cG_n \subseteq \cH_n$; see Figure~\ref{fig:hull-picture} for a simulation picture.
If  $d=2$, $\Sigma =I$, and $\mu \neq \0$,
another application of the ideas of the present paper (see Theorem~\ref{thm:com} below) shows that
\begin{equation}
\label{eq:centre-of-mass-area-lil}
 \limsup_{n \to \infty}  \frac{A ( \cG_n )}{   n^{3/2} \sqrt{ \log \log n }} =  \vartheta \| \mu \| , \as  ,
\end{equation}
where $\vartheta \in (0,\infty)$. It would be  interesting to evaluate $\vartheta$;
presently, we do not have a solution to the variational problem that
characterizes $\vartheta$, which seems to require some new ideas. 
In Proposition~\ref{prop:kappa-bound} in Appendix~\ref{sec:com} below, we show that $\vartheta \geq 0.090435$;
for comparison with~\eqref{eq:drift-area-lil}, note that $1/\sqrt{6} \approx 0.40825$.

\begin{problem}
Compute the constant~$\vartheta$ in~\eqref{eq:centre-of-mass-area-lil}.
\end{problem} 

\begin{figure}[!h]
\centering
\includegraphics[width=0.99\textwidth]{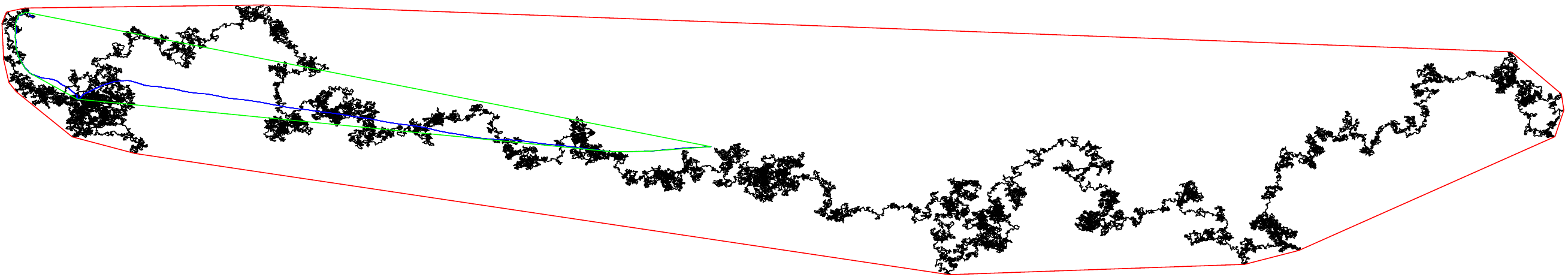}
\caption{Simulated trajectory of $5 \times 10^4$ steps of a planar random walk with non-zero drift, together with its convex hull (red).
Also plotted are the trajectory (blue) of the centre of mass process of the walk defined via~\eqref{eq:com-def},
 and the corresponding convex hull (green).}
\label{fig:hull-picture}
\end{figure}

Our main iterated-logarithm result for $\cH_n$ includes not only~\eqref{eq:drift-area-lil}, but results
for all dimensions~$d$ and all intrinsic volumes, 
and permits general $\Sigma$. In most cases, however, unlike~\eqref{eq:drift-area-lil}, we do not have an explicit value for the constant.
Our approach (like Khoshnevisan's) is founded on Strassen's functional law of the iterated logarithm,
modified appropriately to apply to walks with non-zero drift; in our setting, as in Khoshnevisan's, limiting constants can often
be characterized by variational problems, but in only a limited number of instances is the solution known.
We introduce some more notation in order to state our main results.

Let $C$ denote a non-empty, compact, convex subset of $\R^d$, and $V_d (C)$ its volume ($d$-dimensional Lebesgue measure).
If $C_\lambda$, $\lambda >0$, denotes the parallel body $C_\lambda := \{ x \in \R^d : \inf_{y \in C} \| x - y \| \leq \lambda \}$,
then the \emph{Steiner formula} of integral geometry~\cite[\S 4.1]{schneider} says that
\begin{equation}
\label{eq:steiner}
V_d ( C_\lambda ) = \sum_{k=0}^d \lambda^{d-k} \kappa_{d-k} V_k ( C), 
\end{equation}
where $\kappa_d := \pi^{d/2} / \Gamma ( 1+\frac{d}{2})$ is the $d$-dimensional volume of the unit-radius Euclidean ball in~$\R^d$. The quantities $V_k (C)$ on the right-hand side of~\eqref{eq:steiner}
are the \emph{intrinsic volumes} of~$C$. 
In particular, $V_0(C) \equiv 1$, and the 
intrinsic volumes $V_1(C)$ and $V_{d-1}(C)$ are proportional to the \emph{mean width} $w(C)$ and (Minkowski) \emph{surface area} $S(C)$, respectively: specifically, 
$V_1 (C) = \frac{d \kappa_d}{2\kappa_{d-1}} w(C)$ and $V_{d-1} = \frac{1}{2} S(C)$, see e.g.~\cite[p.~104]{gruber}.

Consider the random variables $V_k (\cH_n)$, $k \in \{1,\ldots,d\}$.
The strong law for the convex hull in the case with drift (see Theorem~3.3.3 of~\cite{mcr})
says that if $\Exp \| Z \| < \infty$, then
\begin{equation}
    \label{eq:hull-slln}
\lim_{n \to \infty} n^{-1} \cH_n = \hull \{ \0, \mu \} =: \seg_\mu, \as,
\end{equation}
where the convergence is in the metric space of non-empty compact, convex sets
with the Hausdorff distance.
The limit set $\seg_\mu$ is a line segment of length~$\|\mu\|$. As
can be seen from the Steiner formula~\eqref{eq:steiner} (see also~\cite[p.~7]{css}), for $\mu \neq \0$, the only non-trivial intrinsic volume of $\seg_\mu$ 
is $V_1 (\seg_\mu ) = \| \mu \|$; one has $V_k (\seg_\mu ) =0$ for $k \geq 2$. This means that from the strong law
one obtains the rather limited information that, a.s.,
\begin{equation}
\label{eq:lln-intrinsic-volumes}
 \lim_{n \to \infty} n^{-1} V_1 (\cH_n ) = \| \mu \|, \text{ and } \lim_{n \to \infty} n^{-k} V_k ( \cH_n ) = 0, \text{ for } k \geq 2 ,\end{equation}
where we have used the fact that $V_k$ is continuous and homogeneous of order~$k$ (see Definition~\ref{def:homogeneous} below).

Theorem~\ref{thm:intrinsic-volume-lil}, our first main result, gives more precise information about the a.s.~large-$n$ behaviour of the intrinsic volumes $V_k (\cH_n)$.
Recall that the random walk $S_n$ is \emph{genuinely $d$-dimensional} if $Z$ is not supported by any $(d-1)$-dimensional subspace~\cite[p.~72]{spitzerbook};
under assumption~\eqref{ass:moments}, this is equivalent to $\Sigma$ being (strictly) positive definite. Throughout the paper,
we write $\RP:=[0,\infty)$.

\begin{theorem}
\label{thm:intrinsic-volume-lil}
Suppose that~\eqref{ass:moments} holds with $\mu \neq \0$.
Let $k \in \{1,2,\ldots, d\}$.
Then there exists a constant $\Lambda (d,k,\cL_Z) \in \RP$, depending on $d$, $k$, and the law $\cL_Z$ of $Z$, such that
\begin{equation}
\label{eq:intrinsic-volume-lil}
 \limsup_{n \to \infty} \frac{V_k (\cH_n)}{\sqrt{2^{k-1} n^{k+1} (\log \log n)^{k-1} }} = \Lambda (d,k,\cL_Z), \as 
\end{equation}
Moreover, if $S_n$ is  genuinely $d$-dimensional, then $\Lambda (d,k,\cL_Z) >0$.
\end{theorem}

When $d=1$, the result~\eqref{eq:intrinsic-volume-lil} is already contained in~\eqref{eq:lln-intrinsic-volumes},
so the main interest in Theorem~\ref{thm:intrinsic-volume-lil} is when $d \geq 2$ (see also Remarks~\ref{rems:intrinsic-volume-lil} below).
We do not have, in general,
a formula for the constants $\Lambda (d,k,\cL_Z)$ appearing in~\eqref{eq:intrinsic-volume-lil},
 but we do have a Strassen-type variational characterization
in the case $k=d$ (the volume), in which case $\Lambda(d,k,\cL_Z)$ depends on $\cL_Z$ only via $\mu$ and $\Sigma$,
  through an isoperimetric problem reminiscent of the classical Dido problem, for which we
can provide an explicit solution when $d=2$. To present these results
(Theorem~\ref{thm:constants} below) we need some additional notation.

The symmetric, non-negative definite covariance matrix $\Sigma$
defines a linear transformation of $\R^d$. If $d \geq 2$ and~$\mu \neq \0$ (non-zero drift), we denote by $\Sigmap$ 
the matrix of the linear transformation on~$\R^{d-1}$
induced by the action of $\Sigma$ on the orthogonal complement of the linear subspace generated by~$\mu$.
In other words, if $\Sigma$ is expressed in an orthonormal basis of $\R^d$ that includes $\hat \mu = \mu / \| \mu \|$ as a basis vector, the $(d-1)$-dimensional 
\emph{reduced covariance matrix} $\Sigmap$ is obtained
by deleting the row/column corresponding to $\hat \mu$; note that since $\Sigma = \Exp ( Z Z^\tra ) - \mu \mu^\tra$,
in the definition of $\Sigmap$ 
one could replace $\Sigma$ by $\Exp ( Z Z^\tra )$ and get the same reduced matrix.

For $f : [0,1] \to \R^{d-1}$, $d \geq 2$, define the \emph{space--time convex hull} of~$f$ by
\begin{equation}
\label{eq:H-def} H(f) := \hull \bigl\{ (t, f(t) ) : 0 \leq t \leq 1 \bigr\} \subset \R^d . \end{equation}
Let $U_{d-1}$, $d \geq 2$, denote the set of $f : [0,1] \to \R^{d-1}$, with $f(0) = \0$, 
whose Cartesian components $f_1, \ldots, f_{d-1}$ are absolutely continuous,
and satisfy $\sum_{i=1}^{d-1} \int_0^1 f'_i (s)^2 \ud s \leq 1$.

\begin{theorem}
\phantomsection
\label{thm:constants}
\begin{enumerate}[label=(\roman*)]
\item\label{thm:constants-i}
For $k=d \geq 2$, the  constant $\Lambda (d,k,\cL_Z)$ in~\eqref{eq:intrinsic-volume-lil} is given by 
\begin{equation}
\label{eq:volume-constant}
\Lambda (d,d,\cL_Z ) = \lambda_d \cdot \| \mu \| \cdot \sqrt{ \det \Sigmap }, \end{equation}
where the constants~$\lambda_d \in (0,\infty)$ are given through the variational formula 
\begin{equation}
\label{eq:gamma-def}
\lambda_d := \sup_{f \in U_{d-1}} V_d  (  H(f)  ) . \end{equation}
In the planar case, the constant in~\eqref{eq:gamma-def} takes value $\lambda_2 = \sqrt{3} / 6$.
\item\label{thm:constants-ii}
For $k=1$,  the  constant $\Lambda (d,k,\cL_Z)$ is given by $\Lambda (d,1,\cL_Z) = \| \mu\|$.
\end{enumerate}
\end{theorem}

\begin{remarks}
	\phantomsection
\label{rems:intrinsic-volume-lil}
\begin{myenumi}[label=(\alph*)]
\item 
If $S_n$ is genuinely $d$-dimensional, then $\det \Sigmap >0$.
\item 
\label{rems:intrinsic-volume-lil-a}
In the case $k=d \geq 2$, the statement~\eqref{eq:intrinsic-volume-lil} 
combined with~\eqref{eq:volume-constant} shows that the volume of the convex hull satisfies the LIL
\begin{equation}
\label{eq:volume-lil}
 \limsup_{n \to \infty} \frac{V_d (\cH_n)}{\sqrt{2^{d-1} n^{d+1} (\log \log n)^{d-1} }} = \lambda_d \cdot \| \mu \| \cdot \sqrt{ \det \Sigmap },\end{equation}
where $\lambda_d$ is given by~\eqref{eq:gamma-def}. When $d=2$ and $\Sigma = I$, this plus the fact that $\lambda_2 = \sqrt{3}/6$ 
gives our LIL for the area quoted at~\eqref{eq:drift-area-lil} above.
\item
\label{rems:intrinsic-volume-lil-b}
 We do not, in general, have the solution to the variational problem~\eqref{eq:gamma-def}
for $d \geq 3$. The claim in Theorem~\ref{thm:constants} that $\lambda_2 = \sqrt{3} / 6$ is established by solving the $d=2$ variational problem in~\eqref{eq:gamma-def}. The solution is presented in Theorem~\ref{thm:planar-optimum} below.
\item
\label{rems:intrinsic-volume-lil-bb}
Generically, LILs are closely linked to large deviations;  recent
results on large deviations for 
planar random walks 
can be found in~\cite{av,vysotsky}. In particular, in~\cite{av,vysotsky}
the large-deviations rate function
for a large class of planar random walks with Gaussian increments
are computed, and these results provide an alternative route to extracting the constant~$\lambda_2 = \sqrt{3} / 6$: see Remark~\ref{rem:vv} below.
\item
\label{rems:intrinsic-volume-lil-c} 
In the case $k=1$, the strong law in~\eqref{eq:lln-intrinsic-volumes}
demonstrates that the `$\limsup$' in~\eqref{eq:intrinsic-volume-lil} is in fact a limit (with the constant identified in Theorem~\ref{thm:constants}\ref{thm:constants-ii}).
\end{myenumi}
\end{remarks}

Here we give a brief overview of the remaining part of the article. 
We obtain our LIL for intrinsic volumes, Theorem~\ref{thm:intrinsic-volume-lil},
from a functional law of the iterated logarithm for convex hulls of random walks with drift, in the vein of (and proved using) Strassen's celebrated
functional law of the iterated logarithm~\cite{strassen}. Our result is a companion to a LIL of Khoshnevisan~\cite{khoshnevisan}
which applies to functionals of convex hulls of \emph{zero-drift} random walks.  
In Section~\ref{sec:strassen} we  review the Strassen-type theorem that we will need (Theorem~\ref{thm:strassen-rw})
 and present  a small extension of Khoshnevisan's result (Theorem~\ref{thm:khoshnevisan}) to permit general $\Sigma$.
In Section~\ref{sec:strassen-drift} we present analogues of the Strassen and Khoshnevisan results in the case of non-zero drift,
and give the proof of Theorem~\ref{thm:constants}, up to the evaluation of the constant~$\lambda_2$. 
The classical normalization in the LIL is the Khinchin scaling function $\sqrt{2n \log \log n}$;
since our walks have non-zero drift,  we  instead scale the drift direction linearly (order $n$), and the other coordinates with the Khinchin scaling: see~\eqref{eq:psi-def} below.
The proof of Theorem~\ref{thm:intrinsic-volume-lil} needs some additional
work to overcome the fact that most intrinsic volumes do not scale simply under our non-isotropic scaling transformation:
the main additional ingredient is a zero--one law for functionals of convex hulls of random walks with non-zero drift.
This zero--one law is the subject  of Section~\ref{sec:zero-one}, which also contains the proof of Theorem~\ref{thm:intrinsic-volume-lil}.
Section~\ref{sec:planar} is devoted to the evaluation of the constant $\lambda_2$ from Theorem~\ref{thm:constants},
by solving a planar isoperimetric problem similar to the classical Dido problem.

\section{Iterated-logarithm laws of Strassen and Khoshnevisan}
\label{sec:strassen}

In this section we  review  Strassen's functional LIL, first
in the case of $d$-dimensional Brownian motion (Theorem~\ref{thm:strassen-bm}) and then
for zero-drift random walk (Theorem~\ref{thm:strassen-rw}). We then present  (a slight generalization of) a theorem of  Khoshnevisan
of functionals of convex hulls of random walks (Theorem~\ref{thm:khoshnevisan}).
As a first application, we deduce a shape theorem extending one from~\cite{mcrw} (Corollary~\ref{cor:shape}).
It is worth pointing out that the LILs of Strassen and Khoshnevisan were both partially anticipated
 by a spectacular~1955 paper of L\'evy~\cite{levy}.

Let $\cCd$ denote the set of continuous $f : [0,1] \to \R^d$, and let $\cCdo$ denote the subset of those $f \in \cCd$ for which $f(0)=\0$.
Endowed with the uniform (supremum) metric $\rho_\infty$ defined by $\rho_\infty (f,g) := \sup_{0 \leq t \leq 1} \| f(t) - g(t) \|$, $\cCd$ is a complete metric space.
In components, write $f = (f_1, \ldots, f_d) \in \cCd$, so $f_i \in \cC_1$. If $f_i$ is absolutely continuous, then its derivative $f_i'$ exists a.e. 
If all components of $f$ are absolutely continuous, we say that $f$ is absolutely continuous;
then the (componentwise) derivative $f' := (f_1', \ldots, f_d')$ exists a.e.

For $f \in \cCd$, set
\[ \| f \|^2_2 := \int_0^1 \| f (s) \|^2 \ud s
= \  \sum_{i=1}^d \int_0^1  f_i (s) ^2 \ud s 
 ,\]
which  defines the $L^2$ norm~$\| \, \cdot \, \|_2$. 
Let $\cA_d$ denote the set of absolutely continuous $f \in \cCdo$
with~$\| f' \|_2 < \infty$.
Then $\cA_d$ is a Hilbert space (the Cameron--Martin space for the Wiener measure~\cite[p.~339]{ry})
with inner product $\langle f, g \rangle_{\cA_d} := \sum_{i=1}^d \int_0^1 f_i' (s) g_i'(s) \ud s$,
and norm 
\begin{equation}
\label{eq:c-m-norm}
 \| f \|_{\cA_d} := \| f' \|_2  = \left( \int_0^1 \| f' (s) \|^2 \ud s \right)^{1/2} . \end{equation}
 The unit ball 
\begin{equation}
\label{eq:unit-ball} U_d := \{ f \in \cA_d : \| f \|_{\cA_d} \leq 1 \} \subset \cCdo \end{equation}
is compact; see~\cite[\S VIII.2]{ry}. 
Let $B \in \cCdo$ be a standard $d$-dimensional Brownian motion.
The Khinchin scaling function for the  classical LIL is
\begin{equation}
\label{eq:ell-def}
 \ell (n) := 1 \text{ for } n \in \{0,1,2\}, \text{ and } \ell (n) := \sqrt{ 2 n \log \log n } \text{ for } n \geq 3 .\end{equation}
For each $n \in \N$, set
\[ B^\star_n(t) := \frac{B (nt)}{\ell (n)} , \text{ for } t \in [0,1] .\]
Strassen's theorem for $d$-dimensional Brownian motion (cf.~\cite{strassen}) states the following.

 \begin{theorem}[Strassen's LIL for Brownian motion]
\label{thm:strassen-bm}
Let $d \in \N$. 
A.s., the sequence $B^\star_n$ in $\cCdo$ is relatively compact, and   its set of limit points is~$U_d$ as defined at~\eqref{eq:unit-ball}.
\end{theorem}

In other words, Theorem~\ref{thm:strassen-bm} states that, a.s., (a)~every subsequence of $B^\star_n$ contains a further subsequence that
converges, its limit being some $f \in U_d$, and~(b) for every $f \in U_d$, there is a subsequence of $B^\star_n$ that converges to~$f$.
For instructive proofs of Theorem~\ref{thm:strassen-bm},
see~\cite[pp.~225--230]{schilling} (for $d=1$) or \cite[pp.~346--348]{ry}, \cite[pp.~21--24]{ds} (for general~$d$).

To state the random-walk version of Strassen's LIL,  for $n\in\ZP$, define
 linear interpolation of the random walk trajectory, parametrized from time $t=0$ to $t=1$, by
\begin{equation}
\label{eq:Y-def}
 Y_n (t) := S_{\lfloor nt \rfloor} + (nt-\lfloor nt \rfloor) Z_{\lfloor nt \rfloor +1}, \text{ for } t \in [0,1].\end{equation}
Then $Y_n \in \cCdo$ for every $n \in \ZP$. 
Note also that $\hull Y_n [0,1] = \cH_n$, since the convex hull of $Y_n [0,1]$
is determined only by $Y_n (k/n)$, $k \in \{0,1,\ldots, n\}$ (linear interpolation  does not affect the convex hull).

The symmetric, non-negative definite matrix $\Sigma$ defined in~\eqref{ass:moments} has a unique symmetric, non-negative definite
square-root $\Sigma^{1/2}$, such that $\Sigma^{1/2} \Sigma^{1/2} = \Sigma$. The matrix $\Sigma^{1/2}$ acts as a linear transformation of $\R^d$
via $x \mapsto \Sigma^{1/2} x$, $x \in \R^d$, and, for $f \in \cCdo$, the function $\Sigma^{1/2} f \in \cCdo$ is given by $(\Sigma^{1/2} f)(t) = \Sigma^{1/2} f(t)$, $t \in [0,1]$.
If~\eqref{ass:moments} holds with $\mu = \0$, then
Donsker's theorem~\cite[p.~393]{stroock} says that $n^{-1/2} Y_n$ converges weakly to $\Sigma^{1/2} B$ as $n \to \infty$.
Strassen's theorem for  $d$-dimensional  random walk (cf.~\cite{strassen}) is as follows.

 \begin{theorem}[Strassen's LIL for random walk]
\label{thm:strassen-rw}
Let $d \in \N$. 
Suppose that~\eqref{ass:moments} holds with $\mu = \0$. 
A.s., the sequence $Y_n / \ell (n)$ in $\cCdo$ is relatively compact, and   its set of limit points is~$\Sigma^{1/2} U_d$.
\end{theorem}

The usual proof of the $d=1$ case of Theorem~\ref{thm:strassen-rw} goes by Skorokhod embedding and
using the Brownian result, Theorem~\ref{thm:strassen-bm} (cf.~\cite[\S 3.5]{stout}). For $d \geq 2$, one substitutes 
a more sophisticated strong approximation argument, such as~\cite{philipp} or~\cite[Thm.~2]{einmahl}; Theorem~\ref{thm:strassen-rw}
is also a consequence of more general results on Hilbert or Banach spaces~\cite[Ch.~8]{lt}.
 
Define the Euclidean distance between points  $x, y \in \R^d$ by $\rho_E (x,y) := \| x - y \|$,
and between a point $x \in \R^d$ and a set $A \subseteq \R^d$ by $\rho_E(x,A) := \inf_{y \in A} \rho_E(x,y)$.
Let $\fKdo$ denote the set of compact subsets of $\R^d$ containing $\0$, endowed with the Hausdorff metric
defined by 
$\rho_H(C_1,C_2) :=  \max  \{ \sup_{x\in C_1}\rho_E(x,C_2),\sup_{y\in C_2}\rho_E (y,C_1) \}$.
By $\fCdo \subset \fKdo$ we denote the set of all~$C \in \fKdo$ that are convex.
If $f \in \cCdo$, then the trajectory $f[0,1]$ (interval image) is in~$\fKdo$,
and $\hull f[0,1]$ is in~$\fCdo$ (cf.~\cite[p.~44]{gruber}).
The following terminology is standard.

\begin{definition}[Homogeneous function]
\label{def:homogeneous}
The function~$F:  \fKdo  \to \R$ is \emph{homogeneous} of order~$\beta \in \RP$ if $F ( \lambda A ) = \lambda^\beta F(A)$ for all $A \in \fKdo$ and all scalar $\lambda \in (0,\infty)$.
\end{definition}

The following geometrical consequence of Theorem~\ref{thm:strassen-rw}
is essentially Proposition~3.2 of Khoshnevisan~\cite{khoshnevisan}, although 
we relax the assumptions in~\cite{khoshnevisan} to permit arbitrary $\Sigma$;
see also~\cite[pp.~465--6]{kl}. 

\begin{theorem}[Khoshnevisan]
\label{thm:khoshnevisan}
Suppose that~\eqref{ass:moments} holds with $\mu = \0$.
Let~$F: (\fKdo, \rho_H) \to (\R, \rho_E)$ be continuous and homogeneous of order~$\beta \in \RP$ (see Definition~\ref{def:homogeneous}).
Then
\[ \limsup_{n \to \infty} \frac{ F ( \cH_n) }{\ell (n)^\beta} = \sup_{f \in U_d} F \left( \Sigma^{1/2} \hull f[0,1] \right), \as \]
\end{theorem}
\begin{proof}
The mapping $f \mapsto \hull f[0,1]$ is continuous from $(\cCdo, \rho_\infty)$ to $(\fKdo,\rho_H)$
(see e.g.~Lemma~2.1 of~\cite{wx2}). Hence $f \mapsto F ( \hull f[0,1])$
is continuous from $(\cCdo, \rho_\infty)$ to $(\R,\rho_E)$. By homogeneity and the fact that $\hull Y_n [0,1] = \cH_n$, we get
\[ \frac{ F ( \cH_n)}{\ell (n)^\beta } = F \left( \frac{\cH_n}{\ell(n)} \right) = F \left( \frac{\hull Y_n [0,1]}{\ell (n)} \right) ,\]
so that the result follows from Theorem~\ref{thm:strassen-rw}.
\end{proof}

This theorem applies to $F=V_k$ an intrinsic volume,
because
 the intrinsic volume function~$V_k$ is 
continuous and 
homogeneous of order $k$~(see e.g.~Theorem~6.13 of~\cite[p.~105]{gruber}).

Khoshnevisan also formulated a version of Theorem~\ref{thm:khoshnevisan} for Brownian motion
(which is deduced from Theorem~\ref{thm:strassen-bm} in the same way as 
Theorem~\ref{thm:khoshnevisan} is obtained from Theorem~\ref{thm:strassen-rw});
 important antecedent results in that setting are due to~L\'evy~\cite{levy} and Evans~\cite[\S 1.3]{evans}.

A further consequence (of either Theorem~\ref{thm:strassen-rw} or Theorem~\ref{thm:khoshnevisan}) 
is the following 
``shape theorem'' for convex hulls of walks with zero drift and finite variance,
which says, roughly speaking, that the process~$\cH_n$ achieves every possible shape infinitely often. This contrasts with the case $\mu \neq \0$, where the strong law~\eqref{eq:hull-slln}
says that the shape converges to a line segment. Corollary~\ref{cor:shape} is an 
extension of Theorem~1.5 of~\cite{mcrw}, which dealt with the case $d=2$,
to general dimensions; the proof of~\cite{mcrw} also readily extends to prove Corollary~\ref{cor:shape}, so we only indicate the argument
here to demonstrate how the Strassen-type results can be applied.  Write $\diam A := \sup_{x,y \in A} \rho_E (x,y)$
for the diameter of~$A \in \fKdo$.

\begin{corollary}[Shape theorem]
\label{cor:shape}
Suppose that~\eqref{ass:moments} holds with $\mu = \0$, and that $\Sigma$ is (strictly) positive definite.
Then for every $C \in \fCdo$ with $\diam C = 1$,
\[ \liminf_{n \to \infty} \rho_H \left( \frac{\cH_n}{\diam \cH_n} , C \right) = 0 , \as\]
\end{corollary}
\begin{proof}
Since $\Sigma$ is full rank, the random walk $S_n$ is genuinely $d$-dimensional.
It holds that $\lim_{n\to\infty} \diam \cH_n = \infty$, a.s.,
which can be seen as a consequence of the classical LIL; see also~\eqref{eq:diam-lil} below.
For $f \in \cCdo$, define $F (f) :=
\hull f[0,1] / \diam f[0,1]$ if $\diam f[0,1] >0$, and $\{ \0 \}$ otherwise. Then $F : ( \cCdo , \rho_\infty) \to ( \fKdo, \rho_H)$ is continuous outside the set~$\{ f : \diam f[0,1] = 0 \}$
(cf.~Lemma~3.6 of~\cite{mcrw}) and this set has Wiener measure zero, 
hence Theorem~\ref{thm:strassen-rw} implies, similarly to in the proof of Theorem~\ref{thm:khoshnevisan},
that, a.s., 
\[ \frac{\cH_n}{\diam \cH_n} \text{ has as it set of limit points } \left\{ \frac{\hull \Sigma^{1/2} f[0,1]}{\diam \Sigma^{1/2} f [0,1]} : f \in U_d \right\}. \]
Any $C \in \fCdo$ can be arbitrarily well-approximated in Hausdorff distance by a polytope $\Sigma^{1/2} \hull \{ x_1, \ldots, x_m \}$
for some  finite set of points $x_i \in \R^d$
(see Theorem~1.8.16 of~\cite{schneider}).
Given $x_1, \ldots, x_m$, set $L_k := \sum_{i=1}^{k-1} \| x_{i+1} - x_i \|$ for $1 \leq k \leq m$, and define $h_m \in \cCdo$
by $h_m ( L_k/L_m ) = x_k$ for $1 \leq k \leq m$, with linear interpolation; in words,
$h_m$ is a polygonal path, parametrized by arc length, that visits $\0, x_1, \ldots, x_m$ in sequence.
Then $h'_m (t) = L_m$ for a.e.~$t \in [0,1]$, so $f_m := h_m / \sqrt{L_m}$ has $f_m \in U_d$, and~$\hull \Sigma^{1/2} f_m[0,1]/\diam \Sigma^{1/2} f_m [0,1]$
approximates $C /\diam (C) = C$ with an error that can be made arbitrarily small.
\end{proof}

One can apply Theorem~\ref{thm:khoshnevisan} to obtain LILs for various functionals: see Appendix~\ref{sec:khoshnevisan} for some examples that
extend slightly some of those from~\cite{khoshnevisan}.

\section{A Strassen-type theorem for the case with drift}
\label{sec:strassen-drift}

In this section we present analogues of
the LIL results of Section~\ref{sec:strassen} for the case of a non-zero drift;
essentially these results are obtained by combining the strong law of large numbers for the drift
with Strassen's law for the coordinates orthogonal to the drift.
To motivate the random-walk result, we first consider the
case of 
Brownian motion with drift in $\R^d$, $d \geq 2$.
Suppose that $X$ is a Brownian motion with drift $\mu \in \R^d \setminus \{ \0\}$ and infinitesimal covariance matrix $\Sigma$, so that $X (t) = \mu t + \Sigma^{1/2} B(t)$ for $t \in \RP$.

Without loss of generality, we suppose coordinates are chosen so that the 
standard orthonormal basis
$(e_1, \ldots e_d)$ of~$\R^d$, $d \geq 2$, has $e_1 = \hat \mu$.
Write $X_i (t) := e_i^\tra  X (t)$. 
For $2 \leq i, j \leq d$, we have $X_i (t) = e_i^\tra \Sigma^{1/2} B(t)$, and so
$\Exp ( X_i (t) X_j (t) ) = e_i^\tra \Sigma e_j$. Let $\Sigmap$ 
denote the matrix obtained from $\Sigma$ by omitting the first row and column, i.e.,
\[ \bigl( \Sigmap \bigr)_{ij} := \Sigma_{i+1,j+1}, \text{ for } 1 \leq i, j \leq d-1 ,\]
which is the (symmetric, non-negative definite) infinitesimal covariance matrix of $X_2, \ldots, X_d$.
Note that $\det \Sigmap$ is the  first principal minor of $\Sigma$.

 For $n \in \N$,
define $\psi_n : \R^d \to \R^d$, acting on $x = (x_1, \ldots, x_d )$, by
\begin{equation}
\label{eq:psi-def}
\psi_n ( x_1, \ldots, x_d ) = \left( \frac{x_1}{n} , \frac{x_2}{\ell(n)} , \ldots, \frac{x_d}{\ell(n)} \right) .
\end{equation}
Let $I_\mu : [0,1] \to \RP$ denote the function $I_\mu (t) = \| \mu \| t$,
and set
\[ W_{d,\mu,\Sigma} := \{  g = (I_\mu, \Sigmap^{1/2} f ) : f \in U_{d-1} \} , \text{ for } d \geq 2. \]
Observe that
if $f \in \cA_{d-1}$, then
(since entries of~$\Sigma$ are uniformly bounded)
$ \Sigmap^{1/2} f  \in \cA_{d-1}$ too; since $I_\mu'(t) = \| \mu \|$,
it follows that $g = ( I_\mu , f ) \in \cA_d$ whenever
$f \in \cA_{d-1}$. In particular, $W_{d,\mu,\Sigma} \subset \cA_d$.
To enable us to include the (somewhat trivial) case $d=1$ in our statements,
we also set $W_{1,\mu,\Sigma} := \{ I_\mu \}$.

\begin{theorem}
\label{thm:strassen-drift-bm}
Suppose that~$\mu \neq \0$. 
A.s., the sequence $( \psi_n ( X(nt) ) )_{t \in [0,1]}$ in $\cCdo$ is relatively compact, and  its set of limit points is~$W_{d,\mu,\Sigma}$.
\end{theorem}
\begin{proof}
The strong law of large numbers (in functional form, see e.g.~Theorem~3.4 in~\cite{lmw})
says that
$( X_1 (nt)/n )_{t \in [0,1]}$ converges a.s.~in $(\cCdo, \rho_\infty)$ to~$I_\mu$ as $n \to \infty$.
 On the other hand,
Strassen's theorem for Brownian motion on $\R^{d-1}$ (Theorem~\ref{thm:strassen-bm}) implies that
$( (X_2 (nt), \ldots, X_d (nt) )/ \ell(n) )_{t \in [0,1]}$ has as its limit points $\Sigmap^{1/2} U_{d-1}$.
\end{proof}

Here is the random-walk formulation. Recall the definition of $Y_n \in \cCdo$ from~\eqref{eq:Y-def}.

\begin{theorem}
\label{thm:strassen-drift-rw}
Suppose that~\eqref{ass:moments} holds with $\mu \neq \0$.
A.s., the sequence $\psi_n ( Y_n )$ in $\cCdo$ is relatively compact, and   its set of limit points is~$W_{d,\mu,\Sigma}$.
In particular, if $F : (\cCdo, \rho_\infty) \to (\R,\rho_E)$ is continuous, then $\limsup_{n \to \infty} F ( \psi_n ( Y_n ) ) = \sup_{h \in W_{d,\mu,\Sigma}} F( h)$, a.s.
\end{theorem}
\begin{proof}
The proof runs along the same lines as that of Theorem~\ref{thm:strassen-drift-bm},
but in place of Theorem~\ref{thm:strassen-bm} one applies its random-walk analogue, Theorem~\ref{thm:strassen-rw}.
\end{proof}

As a corollary, we obtain an analogue of Khoshnevisan's theorem (Theorem~\ref{thm:khoshnevisan}) in the case with drift.

\begin{corollary}
\label{cor:strassen-hull-drift}
Suppose that~\eqref{ass:moments} holds with $\mu \neq \0$.
Let~$F: (\fKdo, \rho_H) \to (\R, \rho_E)$ be continuous.
Then
\[ \limsup_{n \to \infty} F ( \psi_n ( \cH_n ))  = \sup_{h \in W_{d,\mu,\Sigma}} F ( \hull h [0,1] ), \as \]
\end{corollary}

As another application  of Theorem~\ref{thm:strassen-drift-rw}, we give a LIL 
for the volume of the convex hull $\cG_n$ of the centre of mass process defined at~\eqref{eq:com-def},
which includes the case $d=2$ stated at~\eqref{eq:centre-of-mass-area-lil} above. Define
$\fg : ( \cCdo , \rho_\infty ) \to (\cCdo , \rho_\infty )$ for all $f \in \cCdo$ by
\[ \fg_f (0) := \0, \text{ and } \fg_f (t) := \frac{1}{t} \int_0^t f(s) \ud s  \text{ for } 0 < t \leq 1.\]
 One verifies that $\fg_f$ is indeed in $\cCdo$, since $\lim_{t \downarrow 0} \fg_f (t) = f(0) = \0$.

\begin{theorem}[LIL for the convex hull of the centre of mass]
\label{thm:com} 
Suppose that~\eqref{ass:moments} holds with $\mu \neq \0$. Then 
\[ \limsup_{n \to \infty} \frac{V_d (\cG_n)}{\sqrt{2^{d-1} n^{d+1} ( \log \log n)^{d-1} }} = \tlambda_d \cdot \frac{\| \mu \|}{2} \cdot \sqrt{ \det \Sigmap } , \as,
\]
where the constants $\tlambda_d$ are given via the variational formula 
\begin{equation}
\label{eq:tlambda-def}
\tlambda_d := \sup_{f \in U_{d-1}} V_d ( H ( \fg_f ) ) .\end{equation}
\end{theorem}
 
We defer the proof of Theorem~\ref{thm:com} to Section~\ref{sec:com}, since the proof runs along similar lines to
parts of the proofs of our main results.
We end this section with the proofs of our LIL  for $V_d (\cH_n)$, Theorem~\ref{thm:constants}.

\begin{proof}[Proof of Theorem~\ref{thm:constants}.]
By the scaling property of volumes, for $n \in \N$,
\[ V_d ( \psi_n ( \cH_n ) ) = 
\frac{V_d ( \cH_n)}{n \ell(n)^{d-1}} .\] 
Then, by Corollary~\ref{cor:strassen-hull-drift} with $G = V_d$,
we obtain, a.s., 
\[ \limsup_{n \to \infty} \frac{V_d (\cH_n)}{\sqrt{2^{d-1} n^{d+1} (\log \log n)^{d-1} }} = \sup_{h \in W_{d,\mu,\Sigma}} V_d ( \hull h[0,1] ) .\]
Now $h \in  W_{d,\mu,\Sigma}$ has $h =  (I_\mu, \Sigmap^{1/2} f ) $ for some $f \in U_{d-1}$, and, by scaling, if $h_0 := ( I_{e_1}, f)$,
\[  V_d ( \hull h[0,1] ) = \| \mu\| \cdot \det \Sigmap^{1/2} \cdot  V_d ( \hull h_0 [0,1] ) 
= \| \mu\| \cdot \sqrt { \det \Sigmap } \cdot  V_d ( H (f)  )
,\]
where $H$~is defined at~\eqref{eq:H-def}. Thus we deduce~\eqref{eq:volume-lil} and characterize the constant~$\lambda_d$ via~\eqref{eq:gamma-def}.
This completes the proof of part~\ref{thm:constants-i} of Theorem~\ref{thm:constants}.
For part~\ref{thm:constants-ii} of the theorem, 
we simply note from comparison of the strong law of large numbers in~\eqref{eq:lln-intrinsic-volumes} 
with~\eqref{eq:intrinsic-volume-lil} in the case $k=1$, we identify
that $\Lambda (d,1,\cL_Z) = \| \mu\|$.
\end{proof}

The proof of Theorem~\ref{thm:intrinsic-volume-lil} needs more work, because for $k \neq d$ the functional $V_k$ 
does not behave so nicely under the scaling operation $\psi_n$.
The idea is to use the Strassen-type result~Theorem~\ref{thm:strassen-drift-rw}
to obtain upper and lower bounds of matching order, and then conclude using a zero--one law. 
The next section presents this zero--one law (Theorem~\ref{thm:zero-one}), which is of some independent interest,
and then gives the proof of Theorem~\ref{thm:intrinsic-volume-lil}.

\section{A zero--one law for walks with drift}
\label{sec:zero-one}

Define $\cH_\infty := \cup_{n \in \ZP} \cH_n = \hull \{ S_0, S_1, S_2, \ldots \}$.
The event $\{ \cH_\infty = \R^d \}$, that the convex hull asymptotically
fills out all of space, has  probability $0$ or $1$ as a consequence of the Hewitt--Savage zero--one law (see~\cite[p.~335]{lhw}).
Theorem~3.1 of~\cite[p.~7]{mcrw} provides a zero--one law for tail events
associated with $\cH_0, \cH_1, \ldots$ in the case where $\Pr ( \cH_\infty = \R^d ) =1$; 
this follows from a Hewitt--Savage argument based on the fact that for every $k \in \N$ the initial trajectory $S_0, \ldots, S_k$
is eventually enclosed in $\cH_n$ for all $n$ large enough, with probability~1. 
In the case with $\Exp \| Z \| < \infty$,
one has either $\Pr ( \cH_\infty = \R^d ) = 1$ if $\mu = \0$ (Corollary~9.4 in~\cite{lhw}),
or $\Pr ( \cH_\infty = \R^d ) = 0$ if $\mu \neq \0$ (Proposition~4.2 and Theorem~9.2 in~\cite{lhw}).
In particular, in the case where~$\mu \neq \0$, the zero--one law of~\cite{mcrw} does not apply;
see~\cite[\S 9]{lhw} for some further remarks.

The purpose of this section is to provide a zero--one law that can be applied when~$\mu \neq \0$. In this case,
 one cannot hope for a zero--one law that applies to \emph{all}
tail events of the type studied in~\cite{mcrw},
because in the case $\mu \neq \0$, initial steps of the walk are rather likely to remain as vertices of the convex hull for all time, and hence retain influence on values of functionals 
(see Remark~\ref{rem:counterexample} below for one such example).
Our zero--one law will be stated for functionals acting on $\fCdo$, which we call \emph{macroscopic} (Definition~\ref{def:macroscopic} below), having the property that, roughly speaking, changes to the convex set around the origin have relatively negligible influence on the values of the functional when it is sufficiently large; hence the influence of the initial points is lost.

Write $A \sd B := (A \setminus B) \cup (B \setminus A)$ for the symmetric difference of sets $A$ and $B$.
Let $\cB_d (\0, r) := \{ x \in \R^d : \| x \| \leq r\}$ denote the closed $d$-dimensional Euclidean 
ball of radius $r \in \RP$ centred at the origin.
 
\begin{definition}
\label{def:macroscopic} 
We say $F : \fCdo \to \RP$ is \emph{macroscopic} if for every $\eps  \in (0,1)$ and $r \in \RP$
there exists $v \in \RP$ such that
\begin{equation}
\label{eq:macroscopic}
 \left| \frac{ F ( C_1 )}{ F ( C_2 )} - 1 \right| \leq \eps \text{ for all } C_1, C_2 \in \fCdo \text{ with } C_1 \sd C_2 \subseteq \cB_d (\0, r) \text{ and }  F (C_1) \geq v.
  \end{equation}
\end{definition}
\begin{remark}
\label{rem:macroscopic-1}
Note that if $F(C_2) \geq v$ and $| F(C_1) - F(C_2) | \leq \eps F(C_2)$, then $F(C_1) \geq (1-\eps) F(C_2) \geq (1-\eps) v$;
hence one can formulate~\eqref{eq:macroscopic} equivalently with the condition $F (C_2) \geq v$ rather than $F(C_1) \geq v$.
\end{remark}

The macroscopic property in Definition~\ref{def:macroscopic}
 lends itself to our present application (see Proposition~\ref{prop:stability} 
for an explanation), but is also general enough to include  the intrinsic volumes that are our main interest here, as well as a wide
range of other examples, 
as we indicate in the next remark.

\begin{remark}
\label{rem:macroscopic-2}
Let $F : \fKdo \to \RP$ (i.e., acting upon compact but not necessarily convex sets).
Then $F$ is \emph{monotone} if for every $A, B \in \fKdo$ with $A \subseteq B$, it holds that~$F (A) \leq F (B)$,
and $F$ is \emph{subadditive} if for every $A, B \in \fKdo$, $F ( A \cup B ) \leq F (A) + F (B)$.
Take $C_1, C_2 \in \fCdo$ (i.e., convex)
with $C_1 \sd C_2 \subseteq \cB_d (\0,r)$. Clearly $C_1 \subseteq (C_1 \cap C_2) \cup \hull (C_1 \setminus C_2)$,
while $C_1 \setminus C_2 \subseteq C_1$ and so (since $C_1$ is convex) $\hull (C_1 \setminus C_2) \subseteq C_1$.
Hence $C_1 = (C_1 \cap C_2) \cup \hull \cl (C_1 \setminus C_2)$, which expresses the convex compact set $C_1$ as the union of two
convex compact sets (`$\cl$' denotes closure). Then, if $F$ is monotone and subadditive, 
\begin{align}
\label{eq:valuation}
F ( C_1 ) 
& \leq F (C_1 \cap C_2) + F (\hull \cl (C_1 \setminus C_2 )) \leq F (C_2) + F ( \cB_d (\0,r)) .\end{align}
The similar argument with $C_1 , C_2$ interchanged shows that
if $F$ is both monotone and subadditive, then it satisfies the smoothness property
\begin{equation}
\label{eq:smoothness} 
\left| F (C_1) - F (C_2) \right| \leq F ( \cB_d (\0,r) ), \text{ for all } C_1 , C_2 \in \fCdo \text{ with } C_1 \sd C_2 \subseteq \cB_d (\0, r) .
\end{equation}
If~\eqref{eq:smoothness} holds, then
\[ \left| \frac{F (C_1)}{F(C_2)} - 1 \right| \leq \frac{F ( \cB_d (\0,r) )}{F(C_2)} , \text{ for all } C_1 , C_2 \in \fCdo \text{ with } C_1 \sd C_2 \subseteq \cB_d (\0, r) ,
\]
which implies
that $F$~is macroscopic (cf.~Remark~\ref{rem:macroscopic-1}). 
Moreover, observe that in  
the inequality~\eqref{eq:valuation}  the monotonicity and subadditivity of $F$ is required only on \emph{convex} sets;
in particular, it suffices that $F : \fCdo \to \RP$ be a monotone \emph{valuation} (cf.~\cite[p.~110]{gruber} or~\cite[p.~172]{schneider}).
Some particular examples are as follows. 
\begin{itemize}
\item Intrinsic volumes are monotone valuations, hence macroscopic on $\fCdo$.
\item Other monotone valuations include the rotation volumes and rigid motion volumes examined recently in~\cite{lk}.
\item The diameter function $A \mapsto \diam A$ is   monotone and subadditive on $\fKdo$
(for the latter, recall that on $\fKdo$ every set contains~$0$), hence macroscopic on $\fCdo$.
\item If $S_n$ is any transient random walk on $\Z^d$, the associated capacity $\capa : \fKdo \to \RP$ is given by $\capa (A) := \sum_{x \in A \cap \Z^d} \Pr  ( x + S_n \notin A , \text{ for all } n \in \N )$,
and $\capa$ is monotone and subadditive (see e.g.~\cite[\S 25]{spitzerbook}), hence macroscopic.
\item Fix $d =2$ and $\lambda >0$, and consider the function $F_\lambda (A)$  which  counts the 
number of faces of $A \in \fCdo$ whose length exceeds $\lambda$.
Then the function $A \mapsto F_\lambda ( \hull A )$ on $\fKdo$ is neither monotone nor subadditive. We believe $F_\lambda$ is nevertheless macroscopic, but we do not investigate this further here.
\end{itemize}
This list of examples shows that the class of macroscopic functionals contains functions
quite different in nature from intrinsic volumes.
\end{remark}

Here is the formulation of our zero--one law;
we emphasize that we do not assume
 that $\Exp ( \| Z \|^2 ) < \infty$.

\begin{theorem}[Zero--one law]
\label{thm:zero-one}
Let $d \geq 2$. 
Suppose that $\Exp \| Z \| < \infty$ and $\mu \neq \0$, and that~$F : (\fCdo, \rho_H) \to (\RP, \rho_E)$ is macroscopic (see Definition~\ref{def:macroscopic}). 
Suppose also that
\begin{equation}
\label{eq:G-diverges}
\lim_{n \to \infty} F (\cH_n ) = +\infty, \as
\end{equation}
Then~$F ( \cH_n)$ satisfies a zero--one law
in the sense that for every sequence $b_n \in (0,\infty)$, 
\begin{equation}
\label{eq:G-limsup}
\limsup_{n \to \infty} \frac{F (\cH_n)}{b_n} \text{ is a.s.~constant in } [0,+\infty]. \end{equation}
\end{theorem}

\begin{remark}
\label{rem:counterexample}
The condition~\eqref{eq:G-diverges} cannot be removed, in general. To see this, suppose that  $\Exp \| Z \| < \infty$ and $\mu \neq \0$, and
consider the functional $F ( A ) = -\inf_{x \in A} ( \mu^\tra x )$.
Then $F : \fKdo \to \RP$ is  monotone and subadditive, hence macroscopic (see Remark~\ref{rem:macroscopic-2}).
However, the  random variables $Y_n := F(\cH_n)$ satisfy~$\lim_{n \to \infty} Y_n = Y_\infty := -\inf_{x \in \cH_\infty} ( \mu^\tra x )$, a.s.,
which is just $-\inf_{n \in \ZP} T_n$ where $T_n = \mu^\tra S_n$ is a one-dimensional random walk with strictly positive drift. 
As long as $T_n$ is non-degenerate, $Y_\infty$ therefore has a non-trivial distribution, so~\eqref{eq:G-limsup} is violated (for $b_n \equiv b$, constant),
and~\eqref{eq:G-diverges} fails. 
Note that the event $\{ Y_\infty \leq y \}$ is also a tail event for $\cH_n$ in the sense of~\cite[p.~7]{mcrw}.
\end{remark}

For a $d$-dimensional real matrix $M$, the matrix (operator) norm induced by the Euclidean norm is
 $\| M \|_{\rm op} := \sup_{u \in \Sp{d-1}} \| M u \|$.
We need a short calculation.
 For any absolutely continuous $f \in \cCdo$ and any non-negative definite, $d$-dimensional matrix $M$,
we have, adapting~\cite[p.~387]{khoshnevisan},  
\[ M f(t) - M f(s) =  \int_s^t ( M f'(u) ) \ud u, \]
and hence, by the triangle and Jensen inequalities,
\begin{equation}
\label{eq:diameter-inequality}
 \sup_{0 \leq s \leq t \leq 1} \| M f(t) - M f(s) \| 
\leq  \| M \|_{\rm op} \int_0^1 \|  f' (u) \| \ud u 
\leq \| M \|_{\rm op} \sqrt{ \int_0^1 \|  f' (u) \|^2 \ud u  } .\end{equation}
Now we can complete the proof of our LIL for intrinsic volumes, Theorem~\ref{thm:intrinsic-volume-lil}.

\begin{proof}[Proof of Theorem~\ref{thm:intrinsic-volume-lil}.]
Let $k \in \{1,\ldots,d\}$. It suffices to suppose that $S_n$ is genuinely $d$-dimensional (for if not, work instead in $\R^{d'}$ for $d' < d$);
hence $\det \Sigmap >0$. 
We will prove that there exist constants $0 < c_1 < c_2 < \infty$
(where $c_1, c_2$   depend only on $d, k, \mu , \Sigma$) for which
\begin{equation}
\label{eq:intrinsic-volume-lil-3}
c_1 \leq \limsup_{n \to \infty} \frac{V_k (\cH_n)}{\sqrt{n^{k+1} (\log \log n)^{k-1} }} \leq c_2, \as 
\end{equation}
Without loss of generality, choose the basis of $\R^d$ so that $\hat \mu$ is the first coordinate vector.
For $r \in \RP$, define the rectangle $R_d (x, h ; r) := [x, x+h ] \times [ -r, r ]^{d-1}$.
We will apply Theorem~\ref{thm:strassen-drift-rw}, which says that 
$\psi_n ( Y_n )$ in $\cCdo$ is relatively compact, and   its set of limit points is~$W_{d,\mu,\Sigma}$.
Consider $h = ( I_\mu, \Sigmap^{1/2} f )$ with $f \in U_{d-1}$ (so that $h \in W_{d,\mu,\Sigma}$).
Then~\eqref{eq:diameter-inequality} applied to $f \in U_{d-1}$ and with $M =  \Sigmap^{1/2}$
shows that $\sup_{f \in U_{d-1} } \diam ( \Sigmap^{1/2} f [0,1] )$ is bounded by a finite constant.
This means that there exists 
 $r_1 \in \RP$ such that
$\hull h[0,1] \subset R_d (0, \| \mu \| ; r_1)$ for every $h \in W_{d,\mu,\Sigma}$.
Since $W_{d,\mu,\Sigma}$ is the set of limit points for $\psi_n ( Y_n )$, it follows that for every $r > r_1$
there exists a (random, a.s.-finite) $n_0$ such that $\psi_n ( Y_n ) \subseteq R_d ( -1, 2+ \| \mu\| ; r)$
for all $n \geq n_0$, say. Hence by~\eqref{eq:psi-def} we conclude that, a.s.,
\begin{equation}
\label{eq:hull-upper-bound}
 \cH_n \subseteq R_d ( -n, (2+\| \mu \|) n ; r \ell (n) ), \text{ for all but finitely many } n \in \ZP.\end{equation}
On the other hand, it is not hard to see that, provided $\det \Sigmap >0$,
 there exist $r_0 >0$, $\delta_0 \in (0,1/2)$ and some $h \in W_{d,\mu,\Sigma}$ for which
$R_d ( \delta_0 , 1- 2\delta_0 ; r_0 ) \subseteq \hull h [0,1]$. Thus, from Theorem~\ref{thm:strassen-drift-rw},
 for every $\delta \in (\delta_0,1/2)$ and $r \in (0,r_0)$   we have that, a.s.,
\begin{equation}
\label{eq:hull-lower-bound}
 \cH_n \supseteq R_d ( \delta n, n - 2 \delta n  ; r \ell (n) ), \text{ for infinitely many } n \in \ZP.\end{equation}
Finally, the claim~\eqref{eq:intrinsic-volume-lil-3} follows from~\eqref{eq:hull-upper-bound}
and~\eqref{eq:hull-lower-bound} together with the fact that
\begin{equation}
\label{eq:rectangle-volumes}
V_k ( R_d ( x , h ; r ) )= e_{d,k} ( h , 2r, 2r, \ldots, 2 r) = \binom{d-1}{k-1} h (2r)^{k-1} + \binom{d-1}{k} (2r)^k,
\end{equation}
where $e_{d,k}$ is the $k$th elementary symmetric polynomial in $d$ arguments;
one can find~\eqref{eq:rectangle-volumes} as Proposition 5.5 in~\cite{lmnpt}.
This completes the proof of~\eqref{eq:intrinsic-volume-lil-3}.

A consequence of~\eqref{eq:intrinsic-volume-lil-3} is that $\lim_{n \to \infty} V_k ( \cH_n ) = +\infty$, a.s.
By the zero--one law (Theorem~\ref{thm:zero-one}) and the fact that the intrinsic volume functional $V_k$ is macroscopic (see Remark~\ref{rem:macroscopic-2}),
we obtain that there is a constant $\Lambda \in [0,+\infty]$
for which
\begin{equation}
\label{eq:intrinsic-volume-lil-2}
 \limsup_{n \to \infty} \frac{V_k (\cH_n)}{\sqrt{2^{k-1} n^{k+1} (\log \log n)^{k-1} }} = \Lambda, \as 
\end{equation}
Then~\eqref{eq:intrinsic-volume-lil-2} 
combined with~\eqref{eq:intrinsic-volume-lil-3}
shows that $0 < \Lambda < \infty$ (recalling that we have assumed that $S_n$ is genuinely $d$-dimensional).
This completes the proof of~\eqref{eq:intrinsic-volume-lil}.
\end{proof}

The rest of this section is devoted to the proof of Theorem~\ref{thm:zero-one}.
For $k \in \ZP$, let $\Pi_k$ denote the set of all $\pi : \N \to \N$ such that $\pi$ is a permutation on $1,2,\ldots,k$ and $\pi (n) = n$ for all $n > k$.
Then the random walk $S^\pi$ (for $\pi \in \Pi_k$) defined by $S^\pi_0 := \0$ and $S^\pi_n := \sum_{i=1}^n Z_{\pi(i)}$, $n \in \N$,
takes the same increments as $S$, but with a permutation among the first~$k$; note that $S^\pi_n = S_n$ for all $n \geq k$.
Let $\cH^\pi_n := \hull \{ S^\pi_0, \ldots, S^\pi_n \}$.

For $\ell, n \in \ZP$, define 
\begin{equation}
\label{eq:H-k-n-def}
\cH_{\ell,n} := \hull \{ \0, S_\ell, S_{\ell+1} , \ldots, S_n \} \text{ for $n \geq \ell$, and } \cH_{\ell,n} := \{ \0 \} \text{ for } n < \ell,
\end{equation}
and also define
\begin{equation}
\label{eq:H-k-inf-def}
\cH_{\ell,\infty} :=  \bigcup_{n \in \ZP} \cH_{\ell,n} = \hull \{ \0, S_\ell, S_{\ell+1}, \ldots \} .
\end{equation}
Define $\cH^\pi_{\ell,n}$, $n \in \ZP \cup \{ \infty\}$, analogously 
to~\eqref{eq:H-k-n-def}--\eqref{eq:H-k-inf-def}
in terms of $S^\pi$. Then, for every $k \leq \ell$ and all $\pi \in \Pi_k$, it holds that
$\cH_{\ell,n} = \cH^\pi_{\ell,n} \subseteq \cH^\pi_n$ for all $n \in \ZP \cup \{ \infty\}$.

The next proposition 
shows that for every $k \in \ZP$,
  every permutation of $Z_1, \ldots, Z_k$ changes the convex hull $\cH_n$
only within a finite region (with a random radius that depends on $k$ but not on $n$). Indeed,
 it follows from~\eqref{eq:l-k-interpretation} below 
that
\begin{equation}
\label{eq:l-k-interpretation-2}
  \cH_n^{\pi_1} \sd \cH_n^{\pi_2} \subseteq \cB_d (\0, R_k), \text{ for all } n \in \ZP \text{ and all } \pi_1, \pi_2 \in \Pi_k, \as
\end{equation}
In particular, this means that macroscopic functionals 
are asymptotically approximately invariant under $\Pi_k$. This property is the key to our proof of Theorem~\ref{thm:zero-one}.
Theorem~6.2.2 of McRedmond~\cite{mcr}  established a  
result along somewhat similar lines for the space-time convex hull of a one-dimensional random walk.

\begin{proposition}
\label{prop:stability}
Let $d \geq 2$. 
Suppose that $\Exp \| Z \| < \infty$ and $\mu \neq \0$. 
A.s., it holds for every $k \in \ZP$ that
there exists $R_k < \infty$ for which
\begin{equation}
\label{eq:l-k-interpretation}
  \cH_n^\pi \setminus \cH_{k,n}^\pi \subseteq  \cB_d  (\0, R_k), \text{ for all } n \in \ZP \text{ and all } \pi \in \Pi_k.
\end{equation}
\end{proposition}

A crucial geometrical ingredient in the proof of Proposition~\ref{prop:stability} is provided by the next lemma,
which says that $\cH_n$ eventually contains cylinders with axis in the $\hat \mu$ direction of arbitrary radius and height.
Define the cylinder $C_{d} (x,h ; r) := [x,x+h] \times \cB_{d-1} (\0, r) \subset \R^d$.
Denote $\bbO_\mu := \{ u \in \Sp{d-1} : u^\tra \mu = 0\}$
for the sphere (a copy of $\Sp{d-2}$) orthogonal to $\mu$. 

\begin{lemma}
\label{lem:cylinders}
Let $d \geq 2$. 
Suppose that $\Exp \| Z \| < \infty$ and $\mu \neq \0$. 
Fix $r , h \in \RP$ and $k \in \N$. Then, a.s., there exists $x \in \RP$ such that,
for all but finitely many $n \in \ZP$, $C_d ( x, h ; r ) \subseteq \cH_{k, n}$.
\end{lemma}
\begin{proof}
For $x \in \R^d$, define $x^\per \in \R^{d-1}$ to be the projection of~$x$
onto the $(d-1)$-dimensional subspace orthogonal to $\mu$, 
and then set
 $\cH^\per_{k,n} := \hull \{ 0, S_k^\per, S_{k+1}^\per , \ldots, S_n^{\per} \}$ for $n \geq k$.
Here $S^\per = ( S^\per_n )_{n \in \ZP}$ is a random walk on $\R^{d-1}$
with mean drift $\Exp [ Z^\per ] = \0$ in $\R^{d-1}$. 
Corollary~9.4 in~\cite{lhw} implies that
$\cH^\per_{k, \infty} := \cup_{n \geq k} \cH^\per_{k,n}$ satisfies $\Pr ( \cH^\per_{k, \infty} = \R^{d-1} ) = 1$ for all $k \in \ZP$.

Fix $r , h \in \RP$ and $k \in \N$. Then there exists an a.s.-finite $n_0 \in \ZP$ (a stopping time) such that
$\cB_{d-1} (\0, r) \subseteq \cH^\per_{k,n}$ for all $n \geq n_0$.
Then set $x := \max_{0 \leq n \leq n_0} \hat \mu^\tra S_n$ (so $x \in \RP$ is a.s.~finite).
Since, by the strong law of large numbers,
 $\lim_{n \to \infty} \hat \mu^\tra S_n = \infty$, a.s.,
there exists an a.s.-finite $n_1 \geq n_0$ such that $\hat\mu^\tra S_n > x + h$ for all $n \geq n_1$.
There also exists an a.s.-finite $r_1 > r$ such that
$\cH^\per_{k,n_1} \subseteq \cB_{d-1} (\0, r_1)$. 
Furthermore, there exists an a.s.-finite $n_2 \geq n_1$ such that $\cB_{d-1} (\0, 1+r_1 ) \subseteq \cH^\per_{k,n}$ for all $n \geq n_2$.
Therefore, $\cH^\per_{k,n_1}$ is contained in the interior of $\cH^\per_{k,n_2}$,
and so $\cH^\per_{k,n_2} = \hull \{ S^\per_{n_1 +1}, S^\per_{n_1+2}, \ldots, S^\per_{n_2} \}$.

The preceding argument reveals two properties of $\cH_{k,n_2}$. First,
since $\cB_{d-1} (\0, r) \subseteq \cH^\per_{k,n_0} \subseteq \cH^\per_{k,n_2}$ and every $z \in \cH^\per_{k,n_0}$ has $\hat \mu^\per z \leq x$,
it holds that
\begin{equation}
\label{eq:cylinder-left-end}
\text{ for every }  u \in \bbO_\mu, \text{ there exists } s \leq x \text{ such that } s \hat \mu + r u \in \cH_{k,n_2}.
\end{equation}
Second, since $\cB_{d-1} (\0, 1+r_1 ) \subseteq \cH^\per_{k,n_2} = \hull \{ S^\per_{n_1 +1}, S^\per_{n_1+2}, \ldots, S^\per_{n_2} \}$,
where $r_1 > r$, 
and every $n \geq n_1$ has  $\hat \mu^\per S_n \geq x+h$,
it holds that
\begin{equation}
\label{eq:cylinder-right-end}
\text{ for every }  u \in \bbO_\mu, \text{ there exists } s \geq x+h \text{ such that } s \hat \mu + r u \in \cH_{k,n_2}.
\end{equation}
It follows from~\eqref{eq:cylinder-left-end} and~\eqref{eq:cylinder-right-end}, by convexity, that $C_d ( x, h ; r ) \subseteq \cH_{k,n_2}$.
\end{proof}

\begin{proof}[Proof of Proposition~\ref{prop:stability}.]
Fix $k \in \ZP$ (which we suppress in some of the subsequent notation).
For $n \in \N \cup \{ \infty \}$, $x \in \RP$ and $u \in \bbO_\mu$, define
\[ r_n (x , u) := \sup \{ r \in \RP : x \hat \mu +  u r \in \cH_{k,n} \} ,\]
where $\cH_{k,n}$ is given by~\eqref{eq:H-k-inf-def}.
Furthermore, for $r \in \RP$ define
\[ X_r := \inf \bigl\{ x \in \RP : \inf_{u \in \bbO_\mu} r_\infty (x,u) \geq r \bigr\} .\]
A consequence of Lemma~\ref{lem:cylinders} is that
\begin{equation}
\label{eq:cross-section-time}
\Pr ( X_r < \infty ) = 1, \text{ for every } r \in \RP.
\end{equation}
Since, for fixed $k$, $\cH_{k,n}$ is a non-decreasing sequence (in $n$) of compact sets,
if $X_r < \infty$, there is a (random but a.s.~finite) time $\tau_r >k$ such that, for every $n \geq \tau_r$, the convex hull $\cH_{k,n}$ contains $X_r \hat \mu + r \bbO_\mu$, a sphere
in the $(d-1)$-dimensional hyperplane orthogonal to $\mu$ at distance $X_r$ from the origin in the $\hat \mu$ direction.
In particular, $r_n ( X_r , u ) \geq r$ for
all $n \geq \tau_r$ and all $u \in \bbO_\mu$.

Set
$T :=  \sum_{i=1}^{k}  \| Z_i \|$.
Consider the (a.s.-finite) $X_{2T}$ and associated time $\tau_{2T} > k$ with the property described in the preceding paragraph. 
Then $\cH_{k,n} \subseteq \cH_n$ 
contains the sphere $X_{2T} \hat \mu + 2T \bbO_\mu$ for every $n \geq \tau_{2T}$.
Moreover, the random variables $X_{2T}$ and $\tau_{2T}$ are invariant
under permutations $\pi \in \Pi_k$ of $Z_1,\ldots, Z_k$
(since $T$ is invariant, and $\tau_r$ is defined in terms of $\cH_{k,n}$).
Hence $\cH^\pi_n$ 
contains the sphere $X_{2T} \hat \mu + 2T \bbO_\mu$ for every $n \geq \tau_{2T}$ and every $\pi \in \Pi_k$.

 For $n \geq \tau_{2T}$, let $\Gamma_n := \{ X_{2T} \hat \mu + u r_n (X_{2T}, u) : u \in \bbO_\mu \}$,
which is the relative boundary of $\cH_{k,n}$ in the hyperplane orthogonal to $\hat\mu$ at distance $X_{2T}$.
For every $\pi \in \Pi_k$, the points $S^\pi_0, \ldots, S^\pi_k$ are contained in the cylinder $\{ x \hat \mu + T \bbO_\mu : | x|  \leq T \}$. 
Hence for every $z \in \Gamma_n$ and every $j \in \{1,\ldots,n\}$, the angle formed between vectors $S_j^\pi - z$ and $\hat \mu$ (in their common plane) is bounded below by 
\[ \theta_0 = \arctan \left( \frac{T}{T + X_{2T}} \right) > 0 .\]
Every $u \in \bbO_\mu$ has $\Exp [ u^\tra Z ] =\0$, and so the strong law of large numbers
says that, as $n \to \infty$, $n^{-1} u^\tra S_n \to 0$, a.s. This holds a.s.~simultaneously for every $u$ in a countable dense subset of $\bbO_\mu$, and hence
$n^{-1} \sup_{u \in \bbO_\mu} ( u^\tra S_n ) \to 0$, a.s., as $n \to\infty$. On the other hand, a.s., $n^{-1} \mu^\tra S_n \to \| \mu \|^2 >0$.
Thus
\[ \lim_{n \to \infty} \frac{\sup_{u \in \bbO_\mu} ( u^\tra S_n )}{\mu^\tra S_n} = 0 , \as  \]
Consequently, 
there exists a random time $\ell_k \in \N$ such that
$S_{n}$ stays in a cone with vertex at some $x \hat \mu$ ($x \in \RP$)
and angular span $\theta_0/2$, say, for all $n \geq \ell_k$. 

It follows that, for every $\pi \in \Pi_k$ and every $n \geq \ell_k$,
no $S_0^\pi ,\ldots, S_k^\pi$ can be included in any face of $\cH_n$ which also
includes some member of $\{ S_{\ell_k}, S_{\ell_k+1}, \ldots \}$.
Let $F^\pi_n$ denote those faces of $\cH^\pi_n$ which include  at least one of $S_0^\pi ,\ldots, S_k^\pi$,
and $G^\pi_n$ denote those faces of $\cH^\pi_n$ which use none of $S_0^\pi ,\ldots, S_k^\pi$,
 and set $R_k := \sum_{i=1}^{\ell_k} \| Z_i \|$. Since faces in $F^\pi_n$ can include only $S_0^\pi, \ldots, S_{\ell_k}^\pi$,
every point of $F^\pi_n$ lies in $\cB_d ( \0, R_k )$, and  the points
$S_0^\pi ,\ldots, S_k^\pi$ can only appear as vertices in $F^\pi_n$ (and in no other faces of $\cH^\pi_n$).
This means that $\cH_n^\pi \setminus \cB_d (\0, R_k)$ is invariant under $\pi \in \Pi_k$.
Moreover, the vertices of faces in  $G^\pi_n$ are necessarily from $S_{k+1}, S_{k+2}, \ldots$,
which are included in $\cH_{k,n}$. Hence
$\cH^\pi_n \subseteq \cH_{k,n} \cup \cB_d ( \0, R_k )$ for every $\pi \in \Pi_k$.
This completes the proof of~\eqref{eq:l-k-interpretation}.
\end{proof}

\begin{proof}[Proof of Theorem~\ref{thm:zero-one}]
Since $\cH_{k, n} \subseteq \cH_n$, the symmetric difference
$\cH_n \sd \cH_{k, n} = \cH_n \setminus \cH_{k, n}$.
Proposition~\ref{prop:stability} then shows that, for every $k \in \ZP$, a.s., there exists $R_k < \infty$ such that, for every $n \in \ZP$, $\cH_n \sd \cH_{k, n} \subset \cB_d (\0, R_k)$.
Suppose that $F$ is macroscopic. Then, taking $C_1 = \cH_{n}$ and $C_2 = \cH_{k, n}$ in~\eqref{eq:macroscopic}, we see that for every $\eps >0$ there exists $v \in \RP$ (depending on $R_k$) such that
\[
 \left| \frac{ F ( \cH_n )}{ F ( \cH_{k,n} )} - 1 \right| \leq \eps \text{ whenever }   F(\cH_n) \geq v.
  \]
	By~\eqref{eq:G-diverges}, it thus follows that, for every $\eps >0$, a.s.,
\[
 (1-\eps)  \limsup_{n \to \infty} \frac{F ( \cH_{k,n} ) }{b_n}
\leq \limsup_{n \to \infty} \frac{F ( \cH_n ) }{b_n} \leq (1+\eps)  \limsup_{n \to \infty} \frac{F ( \cH_{k,n} ) }{b_n}. \]
Since $\eps>0$ was arbitrary, we conclude that  
\[
\xi := \limsup_{n \to \infty} \frac{F ( \cH_n ) }{b_n} =  \limsup_{n \to \infty} \frac{F ( \cH_{k,n} ) }{b_n}, \as \]
But $F ( \cH_{k,n} )$ is in the $k$-permutable $\sigma$-algebra. 
Thus $\xi$ is also in the $k$-permutable $\sigma$-algebra (extended, as usual, up to a.s.-equivalence). Since $k$ was arbitrary, it follows that~$\xi$
is a.s.~constant in $[0,+\infty]$, by the Hewitt--Savage zero--one law (see~\cite[pp.~232--238]{ct} for details, including the definition  of the extended 
permutable $\sigma$-algebra).
\end{proof}

\section{Evaluating the constant in the planar case}
\label{sec:planar}

\subsection{Solution to the isoperimetric problem}
\label{sec:isoperimetric}

In this section we work with $d=2$, which means our variational
problem concerns functions $f \in \cA_1$, i.e.~(see \S\ref{sec:strassen}) $f:[0,1] \to \R$ that are
absolutely continuous, have $f(0)=0$, and satisfy~$\int_0^1 f'(u)^2 \ud u < \infty$.
We will write $\cA := \cA_1$ for simplicity;
it is, however, helpful to generalize the domain of our functions
from $[0,1]$ to $[0,x]$. 

For $x \in (0,\infty)$, let $\cAx$ denote the set of  $f : [0, x] \to \R$ that are absolutely continuous, have $f(0)=0$,
and satisfy $\int_0^x f'(u)^2 \ud u < \infty$; note that $\cA^1 = \cA$.
Given~$f \in \cAx$, denote by $\of$ and $\uf$ the least concave majorant and greatest convex minorant of $f$ over $[0,x]$, so that
the continuous functions 
$\of : [0,x] \to \R$ and $\uf : [0,x] \to \R$ are such that $\uf (u) \leq f(u) \leq \of(u)$ for all $u \in [0,x]$, $\of (0) = \uf(0) =0$, $\of (x) = \uf (x) = f(x)$, 
$\of$ is concave, $\uf$ is convex, and the difference $\of - \uf$ is minimal. It is easy to see that
\begin{equation}
\label{eq:non-trivial}
 \uf (u ) < \of (u) \text{ for all } u \in (0,x), \text{ unless $f$ is linear.} \end{equation}

For $f \in \cAx$ and Borel $I \subseteq [0,x]$, define functionals
\begin{align}
\label{eq:area-functional} A_{I} (f) & := \int_I  \left( \of (u) - \uf (u) \right) \ud u ,\\
\label{eq:arc-length}
L_I (f) & := \int_I \sqrt{ 1 + f'(u)^2 } \ud u, \\
\label{eq:cost-functional} 
\Gamma_{I} (f)&  := \int_I f'(u)^2 \ud u .\end{align}
Note that we do not include in the notation 
the~dependence on the upper endpoint~$x$ of the domain of $f$. If $I=[a,b]$ is an interval, we write $A_{[a,b]} =A_{a,b}$ for simplicity; similarly for $L_{a,b}$
and $\Gamma_{a,b}$. 
When $x=1$, $\Gamma_{0,1} (f) = \| f \|^2_{\cA}$ as defined at~\eqref{eq:c-m-norm}, 
and $\Gamma_{0,1} (f) \leq 1$ for all~$f \in U_1$. Also, by Cauchy--Schwarz, the arc length functional~$L$
satisfies
\begin{equation}
\label{eq:arc-length-bound} L_{0,1} (f)^2 \leq \int_0^1 ( 1 +f'(u)^2 ) \ud u = 1 + \Gamma_{0,1} (f) ,\end{equation}
  and so $\sup_{f \in U_1} L_{0,1} (f) = \sqrt{2}$ (the supremum being attained by $f(u) = u$).

Here is the main result of this section.

\begin{theorem}
\label{thm:planar-optimum}
Let $x, \gamma \in (0,\infty)$.
 The unique $f \in \cAx$ that maximizes $A_{0,x} (f)$ subject to $\Gamma_{0,x} (f) \leq \gamma$ is $f =  \fopt_{x,\gamma}$ given by
\begin{equation}
\label{eq:f-star-formula} \fopt_{x,\gamma} (u) = \sqrt{\frac{3\gamma}{x^3}} u (x-u) , \text{ for } 0 \leq u \leq x , \end{equation}
which has $\Gamma_{0,x} (\fopt_{x,\gamma}) = \gamma$ and $A_{0,x} (\fopt_{x,\gamma}) = \sqrt{3 \gamma x^3} / 6$. 
\end{theorem}
\begin{remark}
	\phantomsection
\label{rem:vv}
As kindly pointed out to us by Vlad Vysotsky, 
the Strassen-type functional in the variational problem corresponding to Theorem~\ref{thm:planar-optimum}
is exactly 
the large-deviations rate function for a degenerate, non-centred Gaussian distribution. The solution to this isoperimetric problem 
is given in Proposition~2.15 of~\cite{av}. 
The proof (by approximating the degenerate distribution by non-degenerate ones) is omitted, but is it is fully covered in Theorem~1 of the very recent preprint~\cite{vysotsky}.
\end{remark}

The following subsections lay out the proof of Theorem~\ref{thm:planar-optimum}. Before doing so, we make some remarks
on the relation of Theorem~\ref{thm:planar-optimum} to some of the classical problems of isoperimetry. 
\emph{Dido's Problem}
is to find the rectifiable planar curve of a given length to maximize the area of the corresponding convex hull.
In terms of the area and arc-length functions $A_{0,1}$ and $L_{0,1}$ defined at~\eqref{eq:area-functional} and~\eqref{eq:arc-length},
the solution to the Dido problem reduces to the statement (see e.g.~Theorem~1.1 of~\cite{tilli}) that
\[ A_{0,1} (f) \leq \frac{L_{0,1} (f)^2}{2 \pi} , \text{ for all } f \in \cA , \]
with equality attained by the semicircle described by $f(u) = \sqrt{u (1-u)}$. The function $\fopt := \fopt_{1,1}$ has $L_{0,1} (\fopt) = 1 + \frac{\sqrt{3}}{6} \arcsinh ( \sqrt{3} ) \approx 1.380173$
and $A_{0,1} (\fopt) = \sqrt{3}/6 \approx 0.288675$;
the semicircle whose curved boundary has the same arc length has area $\approx 0.303171$.

The fact that the optimal curve in Dido's problem is a semicircle had been known since antiquity,
although proofs that are fully rigorous by modern standards are more recent: see e.g.~\cite{tilli,moran,melzak,schoenberg}
for precise statements and proofs of this result and higher-dimensional analogues, and e.g.~\cite{zalgaller} for a survey of neighbouring results in isoperimetric problems.

\subsection{Preliminaries}
\label{sec:prelim}

We will show that it suffices to prove Theorem~\ref{thm:planar-optimum} for the case $x = \gamma =1$. To do so, we first
 establish some basic results on how the functionals $A$ and $\Gamma$ behave under affine transformations.
Given $f :\RP \to \R$, define the function $f_{a,x} : \RP \to \R$ by
\[ f_{a,x} (u) := a f ( u/x ) . \]

\begin{lemma}
\label{lem:transforms}
Let $a, x \in (0,\infty)$. If $f \in \cA$, then
\[ A_{0,x} ( f_{a,x} ) = a x A_{0,1} ( f ) , \text{ and } \Gamma_{0,x} (f_{a,x} ) = \frac{a^2}{x}   \Gamma_{0,1} (f) .\]
On the other hand, if $f \in \cAx$, then
\[  A_{0,1} ( f_{a,1/x} ) = \frac{a}{x} A_{0,x} ( f ) , \text{ and } \Gamma_{0,1} (f_{a,1/x} ) = a^2 x   \Gamma_{0,x} (f) .\]
\end{lemma}
\begin{proof}
Suppose that $f \in \cA$.
The affine transformation $(u,y) \mapsto (x u, a y)$ 
preserves convexity. Thus if~$\of$ and $\uf$ are the concave majorant and convex minorant, respectively,
$\of_{a,x} = a \of ( u /x)$ and
$\uf_{a,x} = a \uf( u/x)$. Hence 
\begin{align*}
 A_{0,x} ( f_{a,x} ) 
& = \int_0^x \left( \of_{a,x} (u) - \uf_{a,x} (u) \right) \ud u \\
& = a \int_0^x \left( \of (u/x) - \uf (u/x) \right) \ud u \\
& = a x \int_0^1 \left( \of (v) - \uf (v) \right) \ud v = a x A_{0,1} ( f ) .\end{align*}
Moreover, since $f'_{a,x} (u) = \frac{a}{x} f' (u/x)$,
\begin{align*}
\Gamma_{0,x} (f_{a,x} ) & = \int_0^x f'_{a,x} (u)^2 \ud u 
= \frac{a^2}{x^2}  \int_0^x f' (u/x)^2 \ud u \\
& = \frac{a^2}{x}  \int_0^1 f' (v)^2 \ud v = \frac{a^2}{x}   \Gamma_{0,1} (f) .\end{align*}
The other case is similar.
\end{proof}

The next lemma shows that it suffices to consider the case of
$f :[0,1] \to \R$ subject to $\Gamma_{0,1} (f) \leq 1$.

\begin{lemma}
\label{lem:reduction}
Suppose that there exists a  unique $f \in \cA$ that maximizes $A_{0,1} (f)$ subject to $\Gamma_{0,1} (f) \leq 1$; denote this $f$ by $\fopt$.
Then for any $x, \gamma \in (0,\infty)$ there exists a   unique $f \in \cAx$ that maximizes $A_{0,x} (f)$ subject to $\Gamma_{0,x} (f) \leq \gamma$,
 and this $f$ is $\fopt_{x,\gamma}$ given by
\begin{equation}
\label{eq:f-star-general}
 \fopt_{x,\gamma}(u) = \sqrt{ \gamma x} \fopt ( u/x), \text{ for } u \in [0,x] .\end{equation}
Moreover,
\begin{equation}
\label{eq:f-star-general-values}
 \Gamma_{0,x} ( \fopt_{x,\gamma} ) = \gamma, \text{ and } A_{0,x} ( \fopt_{x,\gamma} ) = \sqrt{\gamma x^3} A_{0,1} ( \fopt ) .\end{equation}
\end{lemma}
\begin{proof}
Suppose existence of the optimal $\fopt \in \cA$ is given, as described, with $\Gamma_{0,1} (\fopt) = \gamma \in (0,1]$.
Then $f \in \cA$ defined by $f(u) = \gamma^{-1/2} \fopt (u)$ has 
$\Gamma_{0,1} (f) = 1$ and 
$A_{0,1} ( f ) = \gamma^{-1/2} A_{0,1} (\fopt)$ (as follows from the $x=1$ case of Lemma~\ref{lem:transforms}).
This would contradict optimality of $\fopt$ unless $\gamma =1$; hence $\Gamma_{0,1} (\fopt) = 1$.

Define  $\fopt_{x,\gamma}$ by~\eqref{eq:f-star-general}.
Then, by~\eqref{eq:f-star-general} and Lemma~\ref{lem:transforms}, we have
$\Gamma_{0,x} ( \fopt_{x,\gamma} ) = \gamma \Gamma_{0,1} (\fopt) = \gamma$ (as argued above), and
$A_{0,x} (\fopt_{x,\gamma} ) =  \sqrt{\gamma x^3} A_{0,1} ( \fopt )$, as claimed at~\eqref{eq:f-star-general-values}.
It remains to prove optimality and uniqueness of $\fopt_{x,\gamma}$.

Suppose that there is an $f_0 \in \cAx$
for which $\Gamma_{0,x} (f_0) = \gamma_0 \in (0,\infty)$ with $\gamma_0 \leq \gamma$,
and $A_{0,x} (f_0) \geq A_{0,x} (\fopt_{x,\gamma})$.
Then we can define $f \in \cA$ by $f (u) = (\gamma x)^{-1/2} f_0 ( x u)$ for $u \in [0,1]$.
By Lemma~\ref{lem:transforms}, this $f$ has $\Gamma_{0,1} (f) = \gamma_0 / \gamma \leq 1$ and
\[ A_{0,1} (f) = \frac{1}{\sqrt{\gamma x^3}} A_{0,x} (f_0) \geq \frac{1}{\sqrt{\gamma x^3}} A_{0,x} ( \fopt_{x,\gamma} ) 
= A_{0,1} ( \fopt ) , \]
by~\eqref{eq:f-star-general-values}. This is in contradiction to optimality of $\fopt$ unless $A_{0,x} (f_0) = A_{0,x} ( \fopt_{x,\gamma} )$;
hence $\fopt_{x,\gamma} \in \cAx$ maximizes $A_{0,x}$ subject to $\Gamma_{0,x} (f) \leq \gamma$.
Moreover, suppose $A_{0,x} (f_0) = A_{0,x} ( \fopt_{x,\gamma} )$ so that $A_{0,1} (f) = A_{0,1} (\fopt)$.
Then, by uniqueness, $f = \fopt$, and hence $f_0 = \fopt_{x,\gamma}$, i.e., we get uniqueness of $\fopt_{x,\gamma}$. 
\end{proof}

\subsection{Reduction to non-negative bridges}
\label{sec:reductions}

Lemma~\ref{lem:reduction} shows that to prove Theorem~\ref{thm:planar-optimum},
it suffices to suppose that $x =1$ and $\gamma =1$.
In this subsection we further reduce the class of functions that we need to consider;
the next step is to show that it suffices to work with \emph{bridges}. 
Let $\cB := \{ f \in \cA : f(1) = 0 \}$ (the set of bridges)
and $\cBp := \{ f \in \cB : f(u) \geq 0 \text{ for all } u \in [0,1]\}$ (non-negative bridges). 
Observe that if $f \in \cB$ then $\of \in \cBp$ and $-\uf \in \cBp$, so $A_{0,1} (f) = A_{0,1} (\of)  + A_{0,1} (-\uf)$. Moreover,
if $f \in \cBp$ then $\uf \equiv 0$ and so
\begin{equation}
\label{eq:positive-bridges} A_{0,1} (f) = A_{0,1} (\of) =  \int_0^1 \of (u) \ud u \geq \int_0^1 f(u) \ud u, \text{ for } f \in \cBp , \end{equation}
with equality if and only if $f$ is concave.

For $f \in \cA$, define $\hf : [0,1] \to \R$ by
\begin{equation}
\label{eq:bridge-def}
\hf (u ) := f(u) - u f(1), \text{ for } 0 \leq u \leq 1.
\end{equation}

\begin{lemma}
\label{lem:bridge}
Suppose that $f \in \cA$. Then (i) $\hf \in \cB$; (ii)~$A_{0,1} (\hf) = A_{0,1} (f)$; and~(iii) $\Gamma_{0,1} (\hf) \leq \Gamma_{0,1} (f)$, with equality if and only if $f(1) =0$.
\end{lemma}
\begin{proof}
The function $\og$ defined by $\og (u) := \of (u) - u f(1)$ is concave, since~$\of$ is concave, and satisfies $\hf (u) \leq \og (u)$ for all $u \in [0,1]$.
Moreover, if $g$ is any concave majorant of $\hf$, then the function~$u \mapsto g (u) + u f(1)$ is a concave majorant of~$f$; hence $\og$ must be the least concave majorant of~$\hf$. Similarly,
$\ug(u) := \uf (u) - u f(1)$ is the greatest convex minorant of~$f$. Thus $A_{0,1} ( \hf ) = \int_0^1 (\of (u) - \uf (u)) \ud u = A_{0,1} (f)$, as claimed.
Now, since $f(0)=0$, we have $f(1) =  \int_0^1 f'(u) \ud u$, and hence 
\[ \Gamma_{0,1} ( f ) = \int_0^1  f'(u)^2 \ud u
= \int_0^1 \left( f'(u) - f(1) \right)^2 \ud u + f(1)^2 = \Gamma_{0,1} (\hf ) + f(1)^2, 
\]
so that $\Gamma_{0,1} (\hf ) \leq \Gamma_{0,1} ( f ) $, 
with equality if and only if $f(1) = 0$. In particular, $\hf \in \cA$, and $\hf(1) = 0$ by~\eqref{eq:bridge-def}, so in fact $\hf \in \cB$.
\end{proof}

Subsequent reductions will be based on some surgeries applied to our functions. To describe them, we introduce
notation associated with the \emph{faces} of the convex hull. 

Suppose that $f \in \cB$ is not the constant (zero) function.
Consider the concave majorant $\of$ of $f : [0,1] \to \RP$.
Then $[0,1]$ can be partitioned as $[0,1] = \cE^+ \cup \cF^+$,
where $\cE^+$ are the indices of \emph{extreme points} of $\of$~\cite[p.~75]{gruber}, and
$\cF^+$ are indices of points in the interior of \emph{faces} of $\of$:
that is, $t \in \cF^+$ if and only if there is a unique supporting half-plane
at $(t, \of(t))$ whose boundary $( (t, \of(t)) + v \alpha )_{\alpha \in \R}$, $v \in \R^2$,
intersects $H( f )$ for $\alpha$ in an open interval containing~$0$.

Similarly, we have $[0,1] = \cE^- \cup \cF^-$, where $\cE^-$ and $\cF^-$ correspond to extreme points and faces of $\uf$.
Write
$\cE := \cE^- \cup \cE^+$.
Note that $H (f) = \hull \{ (t, f(t) ) : t \in \cE \}$. Moreover, every $t \in \cE$ has  $\of (t) = f(t)$ (if $t \in \cE^+$) or $\uf (t) = f(t)$ (if $t \in \cE^-$), and,
since $f$ is not linear, the observation~\eqref{eq:non-trivial} shows that $\cE^- \cap \cE^+ \cap (0,1) = \emptyset$,
and so we can speak of $\cE$ as being the set of extreme points, without the possibility of duplications. 
Put differently, every $t \in (0,1)$ must belong to at least one of $\cF^+, \cF^-$, i.e., $\cF^+ \cup \cF^- = (0,1)$.

The set $\cF^+$ is a union of disjoint (maximal) open intervals in $[0,1]$,
each interval corresponding to the horizontal span of a face of $\of$.
 For every $n \in \N$ there are at most finitely many such intervals of length more than $1/n$,
so $\cF^+$ is a union of at most countably many disjoint open intervals. Enumerate the  intervals of $\cF^+$ from left to right by $(u^+_k, v^+_k )$, $k \in I^+_\mathrm{f}$.
Similarly, enumerate the intervals of $\cF^-$ by $(u^-_k, v^-_k )$, $k \in I^-_\mathrm{f}$.
By construction $u_k^\pm, v_k^\pm \in \cE$ for all~$k$.
 For reference, we summarize the preceding
 discussion in the following lemma.

\begin{lemma}
\label{lem:faces}
Suppose $f \in \cB$ is not the zero function. 
There exist two partitions of $[0,1]$, $\cE^+ \cup \cF^+ = \cE^- \cup \cF^- = [0,1]$,
with the following properties.
\begin{itemize}
\item Each set $\cF^\pm$ is the union of pairwise disjoint intervals: $\cF^\pm = \cup_{k \in I^\pm_\mathrm{f}} (u^\pm_k, v^\pm_k )$, 
where $u^\pm_k, v^\pm_k \in \cE^\pm$, it holds that $0 \leq u^\pm_1 < v^\pm_1 < u^\pm_2 < \cdots \leq 1$, and the index set $I^\pm_\mathrm{f}$
is finite or countably infinite.
\item We have $\cF^+ \cup \cF^- = (0,1)$.
\item For all $t \in \cE^+$, $f(t) = \of (t)$, and, for all $t \in \cE^-$, $f(t) = \uf (t)$.
\end{itemize}
\end{lemma}

The following result shows that it suffices to consider non-negative bridges.

\begin{proposition}
\label{prop:symmetrization}
For every $f \in \cB$, there exists $\fs \in \cBp$ for which (i) $\Gamma_{0,1} (\fs ) = \Gamma_{0,1} (f)$; and (ii) $A_{0,1} (\fs) \geq A_{0,1} (f)$.
\end{proposition}

The proof of the Proposition~\ref{prop:symmetrization} uses a kind of \emph{symmetrization}, which is a common idea in isoperimetric problems (see e.g.~the recent work~\cite{bgg,solynin}).

\begin{definition}
\label{def:symmetrization}
Suppose $f \in \cB$, and define the intervals  $(u^+_k, v^+_k )$, $k \in I^+_\mathrm{f}$, and $(u^-_k, v^-_k )$, $k \in I^-_\mathrm{f}$, as in Lemma~\ref{lem:faces}.
Define $f^{\mathrm{s}+} : [0,1] \to \R$ by
\[ f^{\mathrm{s}+} (t ) := \begin{cases} f (u^+_k ) + f (v^+_k) - f (u^+_k +v^+_k - t) & \text{if } t \in [u^+_k, v^+_k], \, k \in I^+_\mathrm{f} \\
f (t) & \text{otherwise.} \end{cases} \]
Similarity, define
\[ f^{\mathrm{s}-} (t ) := \begin{cases} f (u^-_k +v^-_k - t) - f (u^-_k ) - f (v^-_k)& \text{if } t \in [u^-_k, v^-_k], \, k \in I^-_\mathrm{f} \\
- f (t) & \text{otherwise.} \end{cases} \]
\end{definition}

See Figure~\ref{fig:symmetrization} for an illustration: $f^{\mathrm{s}+}$ is constructed by a sequence of reflections
(more accurately, each is a time-reversal and change of sign) across each straight-line face of the concave majorant of~$f$; similarly,
$f^{\mathrm{s}-}$ is derived from the convex minorant.
The proof of Proposition~\ref{prop:symmetrization} amounts to establishing that the $\fs$ in the claim 
therein 
can be chosen to be one of $f^{\mathrm{s}+}$ or $f^{\mathrm{s}-}$ defined at Definition~\ref{def:symmetrization}.
The following series of properties will establish this.

\begin{lemma}
\label{lem:symmetrized-functions}
Suppose $f \in \cB$, and consider the functions  $f^{\mathrm{s}+}$ and $f^{\mathrm{s}-}$ defined at Definition~\ref{def:symmetrization}.
Then (i) $f^{\mathrm{s}+} \in \cBp$ with $f^{\mathrm{s}+} (u) \geq \of (u)$ for all $u \in [0,1]$;
(ii) $f^{\mathrm{s}-} \in \cBp$ with $f^{\mathrm{s}-} (u) \geq -\uf (u)$ for all $u \in [0,1]$;
(iii) $\Gamma_{0,1} ( f^{\mathrm{s}-} ) = \Gamma_{0,1} ( f^{\mathrm{s}+} )= \Gamma_{0,1} (f)$;
  (iv) 
\begin{equation}
\label{eq:area-excess}
 \int_0^1   f^{\mathrm{s}+} (u)  \ud u = \int_0^1 \bigl(  2 \of (u) - f (u) \bigr) \ud u \text{ and } \int_0^1    f^{\mathrm{s}-} (u)   \ud u = \int_0^1 \bigl(   f(u) - 2 \uf (u)  \bigr) \ud u ;\end{equation}
and (v)~$\max( A_{0,1} ( f^{\mathrm{s}+} ) ,  A_{0,1} ( f^{\mathrm{s}-} )  ) \geq  A_{0,1} (f)$.
\end{lemma} 
\begin{proof}
Consider the function $f^{\mathrm{s}+}$.
If $t \in \cE^+$, then $f^{\mathrm{s}+} (t) = f(t) = \of (t)$. Otherwise, $t \in ( u^+_k, v^+_k )$ for some $k \in I^+_\mathrm{f}$.
On this interval, $\of$ is the line segment (face) given by
\[ \of ( t) = f ( u^+_k ) + \left( \frac{t - u^+_k}{v^+_k - u^+_k} \right) \left(  f ( v^+_k ) - f(u^+_k) \right) , \text{ for } t \in ( u^+_k, v^+_k ),
\]
and $f(t) \leq \of (t)$ for all $t \in ( u^+_k, v^+_k )$. Hence, for $t \in ( u^+_k, v^+_k )$, 
\begin{align}
\label{eq:f+excess}
 f^{\mathrm{s}+} ( u^+_k + v^+_k - t ) - \of (  u^+_k + v^+_k - t ) & =   f (v^+_k) - f ( t)
 -  \left( \frac{v^+_k-t}{v^+_k - u^+_k} \right) \left(  f ( v^+_k ) - f(u^+_k) \right)  \nonumber\\
& = \of (t) - f(t) \geq 0 .
  \end{align}
An analogous argument 
	gives
	\begin{align}
\label{eq:f-excess}
 f^{\mathrm{s}-} ( u^-_k + v^-_k - t ) + \uf (  u^-_k + v^-_k - t ) = f(t) - \uf (t) \geq 0 , \text{ for } t \in  ( u^-_k, v^-_k ) .
  \end{align}

Since $f$ is absolutely continuous, the function $f^{\mathrm{s}+}$ is also  absolutely continuous.
Hence the derivative $\frac{\ud}{\ud t} f^{\mathrm{s}+} (t)$ exists for a.e.~$t \in [0,1]$. Since $f^{\mathrm{s}+}$ coincides with $f$ on $\cE^+$,
we have $\frac{\ud}{\ud t} f^{\mathrm{s}+} (t) = f' ( t)$ for a.e.~$t \in \cE^+$.
On the other hand, for $t \in \cF^+$, we have
$\frac{\ud}{\ud t} f^{\mathrm{s}+} (t) = f' (u^+_k +v^+_k - t)$
for a.e.~$t \in (u^+_k, v^+_k)$, and so 
\[ \Gamma_{u^+_k, v^+_k} ( f^{\mathrm{s}+} ) = \int_{u^+_k}^{v^+_k} f' (u^+_k +v^+_k - t)^2 \ud t 
= \Gamma_{u^+_k, v^+_k} ( f ),\]
the last equality obtained using the substitution $s = u^+_k +v^+_k - t$. Hence $\Gamma_{0,1} ( f^{\mathrm{s}+} ) = \Gamma_{0,1} (f)$.
A similar argument shows that $\Gamma_{0,1} ( f^{\mathrm{s}-} ) = \Gamma_{0,1} (f)$, verifying~(iii).
Moreover, the fact that $\Gamma_{0,1} ( f^{\mathrm{s}+} )<\infty$ together with the relation~\eqref{eq:f+excess} implies~(i), and, similarly,
from~\eqref{eq:f-excess} we get~(ii).

Finally,  $f^{\mathrm{s}+} (t) = f(t)= \of (t)$ unless  $t \in ( u^+_k, v^+_k )$ for some $k \in I^+_\mathrm{f}$, where, from~\eqref{eq:f+excess},
\[ \int_{u^+_k}^{v^+_k}  \bigl( f^{\mathrm{s}+} (u) - \of (u) \bigr) \ud u =  \int_{u^+_k}^{v^+_k} \bigl( \of (u) - f(u) \bigr) \ud u ,\]
and then the first formula in~\eqref{eq:area-excess} follows by summing over~$k\in I^+_\mathrm{f}$. Similarly, from~\eqref{eq:f-excess},
\[ \int_{u^-_k}^{v^-_k}  \bigl( f^{\mathrm{s}-} (u) + \uf (u) \bigr) \ud u =  \int_{u^-_k}^{v^-_k} \bigl( f(u) - \uf (u)  \bigr) \ud u ,\]
from which we obtain  the second formula in~\eqref{eq:area-excess}. Since $f^{\mathrm{s}+}$ and $f^{\mathrm{s}-}$ are both in $\cBp$, from~\eqref{eq:positive-bridges} and~\eqref{eq:area-excess} we  conclude that, 
\[ A_{0,1} (  f^{\mathrm{s}+} ) + A_{0,1} (  f^{\mathrm{s}-} ) \geq \int_0^1 \bigl( f^{\mathrm{s}+} (u) + f^{\mathrm{s}-} (u) \bigr) \ud u = 2 \int_0^1 \bigl( \of (u) - \uf (u) \bigr) \ud u = 2 A_{0,1} (f), \]
which implies~(v).
\end{proof}

\begin{figure}[!ht]
\centering
 \begin{tikzpicture}[domain=0:1, scale=5]
\draw[white, line width = 0.2mm]   plot[smooth,domain=0.057:1,samples=100] ({\x},  {0.408-100*(1.057-\x)*(1-(1.057-\x))^2*((1.057-\x)-0.5)*(2-3*(1.057-\x))*(7*(1.057-\x)-1)/((2-(1.057-\x))*(2*(1.057-\x)+3))});
\draw[white, line width = 0.2mm]   plot[smooth,domain=0:0.2584,samples=100] ({\x},  {-0.555-100*(0.258-\x)*(1-(0.258-\x))^2*((0.258-\x)-0.5)*(2-3*(0.258-\x))*(7*(0.258-\x)-1)/((2-(0.258-\x))*(2*(0.258-\x)+3))});
\draw[white, line width = 0.2mm]   plot[smooth,domain=0.295:0.838,samples=100] ({\x},  {-0.342-0.583-100*(0.295+0.838-\x)*(1-(0.295+0.838-\x))^2*((0.295+0.838-\x)-0.5)*(2-3*(0.295+0.838-\x))*(7*(0.295+0.838-\x)-1)/((2-(0.295+0.838-\x))*(2*(0.295+0.838-\x)+3))});
\draw[white, line width = 0.2mm]   plot[smooth,domain=0.89:1,samples=100] ({\x},  {-0.277-100*(1.89-\x)*(1-(1.89-\x))^2*((1.89-\x)-0.5)*(2-3*(1.89-\x))*(7*(1.89-\x)-1)/((2-(1.89-\x))*(2*(1.89-\x)+3))});
\draw[red,line width = 0.20mm] (0.057,0.408) -- (1,0.0);
\draw[red,line width = 0.20mm] (0,0) -- (0.258,-0.555);
\draw[red,line width = 0.20mm] (0.295,-0.583) -- (0.838,-0.342);
\draw[red,line width = 0.20mm] (0.89,-0.277) -- (1,0);
\draw[black,line width = 0.30mm] (0,0) -- (0,-0.15);
\draw[black,line width = 0.30mm] (-0.15,0) -- (0,0);
\draw[black, line width = 0.25mm]   plot[smooth,domain=0:1,samples=100] ({\x},  {100*\x*(1-\x)^2*(\x-0.5)*(2-3*\x)*(7*\x-1)/((2-\x)*(2*\x+3))});
\draw[black,line width = 0.30mm,->] (0,0) -- (1.2,0);
\draw[black,line width = 0.30mm,->] (0,0) -- (0,0.7);
\draw[black,line width = 0.30mm,->] (0,0) -- (0,-0.7);
\end{tikzpicture}
\quad
 \begin{tikzpicture}[domain=0:1, scale=5]
\draw[blue, line width = 0.2mm]   plot[smooth,domain=0.057:1,samples=100] ({\x},  {0.408-100*(1.057-\x)*(1-(1.057-\x))^2*((1.057-\x)-0.5)*(2-3*(1.057-\x))*(7*(1.057-\x)-1)/((2-(1.057-\x))*(2*(1.057-\x)+3))});
\draw[blue, line width = 0.2mm]   plot[smooth,domain=0:0.2584,samples=100] ({\x},  {-0.555-100*(0.258-\x)*(1-(0.258-\x))^2*((0.258-\x)-0.5)*(2-3*(0.258-\x))*(7*(0.258-\x)-1)/((2-(0.258-\x))*(2*(0.258-\x)+3))});
\draw[blue, line width = 0.2mm]   plot[smooth,domain=0.295:0.838,samples=100] ({\x},  {-0.342-0.583-100*(0.295+0.838-\x)*(1-(0.295+0.838-\x))^2*((0.295+0.838-\x)-0.5)*(2-3*(0.295+0.838-\x))*(7*(0.295+0.838-\x)-1)/((2-(0.295+0.838-\x))*(2*(0.295+0.838-\x)+3))});
\draw[blue, line width = 0.2mm]   plot[smooth,domain=0.89:1,samples=100] ({\x},  {-0.277-100*(1.89-\x)*(1-(1.89-\x))^2*((1.89-\x)-0.5)*(2-3*(1.89-\x))*(7*(1.89-\x)-1)/((2-(1.89-\x))*(2*(1.89-\x)+3))});
\draw[red,line width = 0.20mm] (0.057,0.408) -- (1,0.0);
\draw[red,line width = 0.20mm] (0,0) -- (0.258,-0.555);
\draw[red,line width = 0.20mm] (0.295,-0.583) -- (0.838,-0.342);
\draw[red,line width = 0.20mm] (0.89,-0.277) -- (1,0);
\draw[black,line width = 0.30mm] (0,0) -- (0,-0.15);
\draw[black,line width = 0.30mm] (-0.15,0) -- (0,0);
\draw[black, line width = 0.25mm]   plot[smooth,domain=0:1,samples=100] ({\x},  {100*\x*(1-\x)^2*(\x-0.5)*(2-3*\x)*(7*\x-1)/((2-\x)*(2*\x+3))});
\draw[black,line width = 0.30mm,->] (0,0) -- (1.2,0);
\draw[black,line width = 0.30mm,->] (0,0) -- (0,0.7);
\draw[black,line width = 0.30mm,->] (0,0) -- (0,-0.7);
\end{tikzpicture}
\caption{Plot of an example $f \in \cB$ (black line), the faces of $\of, \uf$ corresponding to $\cF$ (red lines), and
the reflection construction of $f^{\mathrm{s}+}$ and $-f^{\mathrm{s}-}$ (blue lines).}
\label{fig:symmetrization}
\end{figure}

\begin{proof}[Proof of Proposition~\ref{prop:symmetrization}.]
Lemma~\ref{lem:symmetrized-functions} shows that we may take $\fs$ to be whichever one of $f^{\mathrm{s}+}$ or $f^{\mathrm{s}-}$ achieves the
maximum in $\max( A_{0,1} ( f^{\mathrm{s}+} ) ,  A_{0,1} ( f^{\mathrm{s}-} )  )$.
\end{proof}

Our final reduction is simple: we can replace a non-negative bridge by its concave majorant.

\begin{lemma}
\label{lem:convexify}
Suppose that $f \in \cBp$. Then $\of \in \cBp$ has $A_{0,1} ( \of ) = A_{0,1} ( f )$ and $\Gamma_{0,1} (\of) \leq \Gamma_{0,1} (f)$.
\end{lemma}
\begin{proof}
Suppose that $f \in \cBp$. 
It follows from Lemma~\ref{lem:faces}
that we can recover $\of$ from $f$ by replacing $f$ over every interval $[u^+_k, v^+_k]$ by 
the straight line segment from $(u^+_k, f(u^+_k))$ to $(v^+_k, f(v^+_k))$.
This preserves the value of the area enclosed by the concave majorant, but can not increase the value of $\Gamma_{0,1}$,
since for $h \in \cA$ and $0 \leq a \leq b \leq 1$, we have $\int_a^b h'(u) \ud u =  h(b)-h(a)$, and 
\[ \Gamma_{a,b} (h) = \int_a^b h' (u)^2 \ud u = \int_a^b \left( h' (u)  - \frac{h(b)-h(a)}{b-a}  \right)^2 \ud u + \frac{(h(b)-h(a))^2}{b-a},\]
which is minimized when $h(u) = h(a) + \left(\frac{u-a}{b-a}\right) ( h(b)-h(a) )$ is the straight line.
\end{proof}

\subsection{Identifying the optimal function}

The reductions in Section~\ref{sec:reductions} show that we can reduce the problem to that of non-negative
bridges that are
their own least concave majorant. The final step is the following.

\begin{lemma}
\label{lem:iso}
Suppose that $f \in \cB$.
Then 
\[ \int_0^1 f(u) \ud u \leq \frac{\sqrt{3}}{6} ,\]
with equality if and only if $f=\fopt$.
\end{lemma}

\begin{remark}
\label{rem:schmidt}
Lemma~\ref{lem:iso}
is closely related to results of Schmidt~\cite{schmidt}, although Schmidt's main results
are concerned with functions with changes of sign; we give a short self-contained proof here, and indicate where a key step
in our argument can be substituted by a secondary result from~\cite{schmidt}. See also~\cite{mpf} (particularly Chapters~IV \& XV) for a host of adjacent results.
\end{remark}

\begin{proof}[Proof of Lemma~\ref{lem:iso}.]
Suppose that $f \in \cB$.
Observe that
\begin{align*}
\int_0^{1/2} f(u) \ud u
& = \int_0^{1/2} \left( \int_0^{u} f' (y) \ud y \right) \ud u = \int_0^{1/2} \left( f'(y) \int_y^{1/2} \ud u \right) \ud y \\
& = \frac{1}{2} \int_0^{1/2} \left( 1 -2 y \right) f' (y) \ud y \\
&  \leq \frac{1}{2} \sqrt {\int_0^{1/2} (1 - 2 y)^2 \ud y \int_0^{1/2} f'(u)^2 \ud u } ,\end{align*}
where the last inequality is Cauchy--Schwarz,
and the condition for equality is that $f'(y)$ and $1 - 2 y$ are linearly dependent. Since $f(0)=0$, that is
\begin{equation}
\label{eq:half-1}
\int_0^{1/2} f(u) \ud u \leq \frac{1}{\sqrt{24}} \left( \Gamma_{0,1/2} (f) \right)^{1/2} , \end{equation}
with equality in~\eqref{eq:half-1} if and only if $f(x) = \theta_0 x (1-x)$ for all $0 \leq x \leq 1/2$ and some $\theta_0 \in \R$.
(The inequality~\eqref{eq:half-1} is the $\lambda =1/2$, $a=1$, $b=2$ case of~(20) in~\cite[p.~306]{schmidt}.)
The same argument with $g(u) = f(1-u)$ shows that
\begin{equation}
\label{eq:half-2}
\int_{1/2}^1 f(u) \ud u \leq \frac{1}{\sqrt{24}} \left( \Gamma_{1/2,1} (f) \right)^{1/2} , \end{equation}
with equality in~\eqref{eq:half-2} if and only if $f(x) = \theta_1 x (1-x)$ for all $1/2 \leq x \leq 1$ and some $\theta_1 \in \R$.
From~\eqref{eq:half-1} and~\eqref{eq:half-2} we obtain
\begin{align}
\label{eq:two-halves}
\int_0^1 f(u) \ud u & = \int_0^{1/2} f(u) \ud u + \int_{1/2}^1 f (u) \ud u \nonumber\\
& \leq \frac{1}{\sqrt{24}} \bigg\{ \sqrt{ \Gamma_{0,1/2} (f) } + \sqrt{ \Gamma_{1/2,1} (f) } \biggr\} .
\end{align}
Suppose that $\Gamma_{0,1} (f) = \gamma \in \RP$. 
Now, by Jensen's inequality, $\frac{\sqrt{\alpha} + \sqrt{\beta}}{2} \leq \sqrt{ \frac{\alpha+\beta}{2}}$, with equality
if and only if $\alpha = \beta$, so we have from~\eqref{eq:two-halves} that
\[ \int_0^1 f(u) \ud u  \leq  \sqrt{\frac{\gamma}{12}} ,\]
with equality if and only if both equality holds in the Cauchy--Schwarz and Jensen inequalities,
i.e., $\Gamma_{0,1/2} (f) = \Gamma_{1/2,1} (f) = \frac{\gamma}{2}$, $f (x) = \theta_0 x (1-x)$ for $x \in [0, \frac{1}{2}]$,
and $f (x) =\theta_1 x (1-x)$ for $x \in [\frac{1}{2},1]$.
Thus in the case of equality, we must have $\theta_0^2 = \theta_1^2 = 3 \gamma$, and $f(x) = \sqrt{3 \gamma} x (1-x)$,
which is $\fopt$ when $\gamma =1$.
\end{proof}

We can now complete the proof of Theorem~\ref{thm:planar-optimum}.

\begin{proof}[Proof of Theorem~\ref{thm:planar-optimum}.]
Lemma~\ref{lem:bridge}, Proposition~\ref{prop:symmetrization}, and Lemma~\ref{lem:convexify}
show that to every $f \in \cA$
one can associate $\tilde f \in \cBp \subset \cA$ that is non-negative and concave, and for which $A_{0,1} (\tilde f) \geq A_{0,1} (f)$ but $\Gamma_{0,1} (\tilde f) \leq \Gamma_{0,1} (f)$.
Hence the optimal $A_{0,1} (f)$ is attained by some $f \in \cBp$ that is non-negative and concave. For such an~$f$,
we have $\of = f$ and $\uf \equiv 0$, so $A_{0,1} (f)$ is simply the integral of~$f$; Lemma~\ref{lem:iso}
then shows that $\fopt$ is optimal.
\end{proof}

\appendix

\section{Proof of the LIL for the centre of mass}
\label{sec:com}

This section deals with the centre of mass process $G_n$ defined by~\eqref{eq:com-def}, and its associated
convex hull $\cG_n = \hull \{ G_0, G_1, \ldots, G_n \}$.
In particular, we prove our LIL in this context, Theorem~\ref{thm:com}, which in the planar case is the statement~\eqref{eq:centre-of-mass-area-lil},
and Proposition~\ref{prop:kappa-bound}, which gives a lower bound on the constant in~\eqref{eq:centre-of-mass-area-lil}.

\begin{proof}[Proof of Theorem~\ref{thm:com}.]
Observe that $f \mapsto \fg_f$ is a (Lipschitz) continuous functional since, for all $f_1, f_2 \in \cCdo$,
\[  \| \fg_{f_1} (t) - \fg_{f_2} (t) \| \leq \frac{1}{t} \int_0^t \| f_1 (s) - f_2 (s) \| \ud s , \text{ for } 0 < t \leq 1, \]
so $\rho_\infty ( \fg_{f_1} , \fg_{f_2} ) \leq \rho_\infty (f_1, f_2)$.
Hence the functional $f \mapsto \hull \fg_f [0,1]$ is continuous from $(\cCdo, \rho_\infty) \to (\fCdo, \rho_H)$.
Recall the definition of the interpolated trajectory~$Y_n$ from~\eqref{eq:Y-def}. 
Theorem~\ref{thm:strassen-drift-rw} implies that 
the sequence  $\psi_n ( \hull \fg_{Y_n} [0,1] )$ in $\fCdo$ is relatively compact, and   its set of limit points is~$\{ H( \fg_h) : h \in W_{d,\mu,\Sigma}\}$.
Every $h \in W_{d,\mu,\Sigma}$ is of the form $h = ( I_\mu , \Sigmap^{1/2} f)$ for $f \in U_{d-1}$,
and then $\fg_h = ( I_{\mu/2} , \Sigmap^{1/2} \fg_f )$. 
We can now argue similarly to in the proof of Theorem~\ref{thm:constants} above to get
\begin{equation}
\label{eq:com-interpolated-LIL}
 \limsup_{n \to \infty} \frac{V_d (\hull \fg_{Y_n} [0,1])}{\sqrt{2^{d-1} n^{d+1} ( \log \log n)^{d-1} }} = \tlambda_d \cdot \frac{\| \mu \|}{2} \cdot \sqrt{ \det \Sigmap } , \as, 
\end{equation}
with $\tlambda_d$ as defined at~\eqref{eq:tlambda-def}.
The LIL at~\eqref{eq:com-interpolated-LIL} does not yet yield the LIL in Theorem~\ref{thm:com},
because, unlike the analogous step in the proof of Theorem~\ref{thm:constants},
there is a discrepancy between $\hull \fg_{Y_n} [0,1]$ (for the interpolated path~$Y_n$)
and $\cG_n$. To quantify the error, write for $n \in \N$ and $t \in [0,1]$, 
\[ \Delta_n (t) := \fg_{Y_n} (t) - G_{\lfloor nt \rfloor} .\]
We claim that
\begin{equation}
\label{eq:error-bounded}
 \limsup_{n \to \infty} \sup_{t \in [0,1]} \| \Delta_n (t) \| < \infty ,\as \end{equation}

The bound~\eqref{eq:error-bounded} will allow us to bound the difference between $V_d  (\hull \fg_{Y_n} [0,1])$ and $V_d (\cG_n)$
using the Steiner formula~\eqref{eq:steiner}; to do so, we need bounds on the lower intrinsic volumes.
At this point we can apply  Theorem~\ref{thm:strassen-drift-rw}, and argue as in the proof of Theorem~\ref{thm:intrinsic-volume-lil} that,
similarly to~\eqref{eq:intrinsic-volume-lil-3}, for each $k \in \{1,2,\ldots, d\}$,
there is a constant $C < \infty$ for which
\begin{equation}
\label{eq:intrinsic-volume-upper-bound-com}
\limsup_{n \to \infty} \frac{V_k (\cG_n)}{\sqrt{n^{k+1} (\log \log n)^{k-1} }} \leq C, \as 
\end{equation}
(Note that in~\eqref{eq:intrinsic-volume-upper-bound-com} we are not claiming that the $\limsup$ is a.s.~constant, and we have not
claimed that the zero--one law in Theorem~\ref{thm:zero-one} applies.)
From~\eqref{eq:error-bounded}, \eqref{eq:intrinsic-volume-upper-bound-com}, and monotonicity of intrinsic volumes,
we obtain from the Steiner formula~\eqref{eq:steiner} that, for every $\eps >0$,
\begin{equation}
\label{eq:error-steiner}
 \lim_{n \to \infty} \frac{|  V_d  (\hull \fg_{Y_n} [0,1]) - V_d (\cG_n) |}{n^{(d/2)+\eps} } = 0 ,\as \end{equation}
Then~\eqref{eq:error-steiner} together with~\eqref{eq:com-interpolated-LIL} yields the LIL in Theorem~\ref{thm:com}.

It remains to verify the claim~\eqref{eq:error-bounded}.
For $t \in [0,1]$ and $n \in \N$, 
\[ \int_0^t Y_n (s) \ud s = \int_0^{\lfloor nt \rfloor/n} Y_n (s) \ud s + \int_{\lfloor nt \rfloor/n}^t Y_n (s) \ud s.\]
For simplicity of notation, write
 $k := \lfloor nt \rfloor$ and $\theta := nt - \lfloor nt \rfloor$,
so that $k \in \{0,1,\ldots, n\}$ and $0 \leq \theta <1$, and $t = \frac{k+\theta}{n}$. Then
\begin{align*}
\int_0^t Y_n (s) \ud s = \sum_{i=0}^{k-1} \int_{i/n}^{(i+1)/n} Y_n (s) \ud s + \int_{k/n}^{(k+\theta)/n} Y_n (s) \ud s 
 , \end{align*}
where, for any $\theta \in [0,1]$, by~\eqref{eq:Y-def},
\[  \int_{i/n}^{(i+\theta)/n} Y_n (s) \ud s = \int_{i/n}^{(i+\theta)/n} \left( S_i + (ns - i) Z_{i+1} \right) \ud s 
= \frac{\theta}{n} S_i + \frac{\theta^2}{2n} Z_{i+1}.\]
It follows that
\begin{align*}
\int_0^t Y_n (s) \ud s = \frac{1}{n} \sum_{i=0}^{k-1} S_i + \frac{1}{2n} S_k 
+ \frac{\theta}{n} S_k + \frac{\theta^2}{2n} Z_{k+1}
= \frac{k}{n} G_k + \frac{2\theta -1}{2n} S_k +  \frac{\theta^2}{2n} Z_{k+1} .\end{align*}
Hence,  for $k \geq 1$,
\[ \fg_{Y_n} (t) = \frac{1}{t} \int_0^t Y_n (s) \ud s 
= \frac{k}{k+\theta} \cdot \frac{n}{k} \int_0^t Y_n (s) \ud s 
= \frac{k}{k+\theta} \left[ G_k +  \frac{2\theta -1}{2k} S_k +  \frac{\theta^2}{2k} Z_{k+1} \right] .\]
Thus we obtain,  by the triangle inequality,
\[ \limsup_{n \to \infty} \sup_{0 \leq t \leq 1} \| \fg_{Y_n} (t) - G_{\lfloor n t \rfloor} \| \leq \| Z_1 \| + \sup_{k \in \N} \left[ 
\frac{\| G_k \|}{k} + \frac{\| S_{k}\|}{2k} + \frac{\| Z_{k+1} \|}{2k} \right]
,\]
which is a.s.~finite, since the strong law implies that $\| S_k \| / k \to \| \mu \|$ and $\| G_k \| / k \to \| \mu \|/2$
(see e.g.~Proposition~1.1 of~\cite{lw}). This verifies~\eqref{eq:error-bounded} and hence completes the proof.
\end{proof}

In the remainder of this section, we establish the following bound on the constant~$\vartheta$ from~\eqref{eq:centre-of-mass-area-lil}.

\begin{proposition}
\label{prop:kappa-bound}
It holds that $\vartheta \geq 0.090435$.
\end{proposition}
\begin{proof}
The $d=2$ case of Theorem~\ref{thm:com} says that, if $\mu \neq \0$,  
\[ \limsup_{n \to \infty} \frac{A (\cG_n)}{\sqrt{2 n^{3}  \log \log n }} = \tlambda_2 \cdot \frac{\| \mu \|}{2} \cdot \sqrt{ \det \Sigmap } ,
\]
where, by~\eqref{eq:tlambda-def}, $\tlambda_2 := \sup_{f \in U_{1}} A ( H ( \fg_f ) )$. Comparison with~\eqref{eq:centre-of-mass-area-lil}
shows that
\begin{equation}
\label{eq:kappa-def}
 \vartheta = \frac{1}{\sqrt{2}} \cdot  \sup_{f \in U_{1}} A ( H ( \fg_f ) ) . \end{equation}
The proof of Proposition~\ref{prop:kappa-bound} 
goes by exhibiting an $f \in U_1$ for which $A ( H ( \fg_f ) ) \geq 0.127894575$.
Define the one-parameter family $f_a : [0,1] \to \R$ ($a^2 \leq 27$) by
\begin{equation}
\label{eq:f-a-def}
 f_a (t) :=  a t^2 \log t - \left( \frac{2a}{3} + \frac{3}{2} \sqrt{ 1 - \frac{a^2}{27} } \right) t^2 + \left( \frac{2a}{3} + \sqrt{ 1 - \frac{a^2}{27} } \right) t .
\end{equation}
Then  the choice $a = 4.059781$ yields the stated bound (we explain below the origin of this choice of parametric family, and the numerical choice of~$a$).
Indeed, the concave majorant of $\fg_f$ for $f = f_a$ with this choice of $a$ coincides with $\fg_f$ over an
 interval $[0,t_0]$ and then follows the straight line to $(1,0)$;
numerical calculation gives $t_0 \approx 0.65213156$ and then $A ( H ( \fg_f ) )  \geq 0.127894575$, as claimed.
\end{proof}

We end this section with an heuristic explanation for how we went about looking
for the $f$ exhibited in the proof of Proposition~\ref{prop:kappa-bound}, as this may prove useful for
readers who seek to improve our bound. 
First note that some calculus yields
\[ \int_0^1 \fg_f (t) \ud t = \int_0^1 f(t) \log (1/t) \ud t = \fg_f (1) + \int_0^1 f'(u) u \log u \, \ud u .\]
While it is not at all clear that the reductions for the analogous problem presented in Section~\ref{sec:reductions}
are applicable in this case, following the reasoning of Lemma~\ref{lem:iso}
suggests trying to choose $f$ to
\begin{equation}
\label{eq:near-optimum-com} 
\text{maximize } \int_0^1 f'(u) u \log u \, \ud u , \text{ subject to } \int_0^1 f(u) \ud u =0 ,\,  \int_0^1 f(u)^2 \ud u \leq 1.\end{equation}
The $f$ which solves~\eqref{eq:near-optimum-com}  is of the form
$f(t) = a t^2 \log t + b t^2 + c t$,
and the two constraints in~\eqref{eq:near-optimum-com}
fix $b$ and $c$, to give the one-parameter family $f_a$ defined at~\eqref{eq:f-a-def}. 
The optimal $f_a$ for~\eqref{eq:near-optimum-com} has $a = 6 \sqrt{  3/7}$, which gives $\int_0^1 \fg_f (t) \ud t = \sqrt{21} /6 \approx 0.12729$. However, 
the $\fg_f$ that results from $f= f_a$ is non-concave, so in fact $A ( H ( \fg_f ) ) \approx 0.12781$. 
This means that~\eqref{eq:near-optimum-com} is not the correct optimization problem to evaluate $\vartheta$; but restricting ourselves to functions $f_a$ of the form given at~\eqref{eq:f-a-def},
we may nevertheless try to maximize $A ( H ( \fg_f ) )$. Rather than $a = 6 \sqrt{  3/7} \approx 3.927922$, numerical experimentation shows that a better choice is $a = 4.059781$,
as we used in the proof of Proposition~\ref{prop:kappa-bound}.

\section{Some extensions of Khoshnevisan's examples for the zero-drift case}
\label{sec:khoshnevisan}

If $Y \sim \cN ( 0 , \Sigma)$ for a symmetric, nonnegative-definite $d$-dimensional matrix~$\Sigma$ with symmetric square-root $\Sigma^{1/2}$,
then 
$u^\tra Y$ is univariate normal
with mean $0$ and variance $u^\tra \Sigma u$.
Since $\Sigma$ is symmetric, its largest eigenvalue is given by 
\begin{equation}
\label{eq:eigenvalue}
  \sigma^2_\star := \sup_{u \in \Sp{d-1}} ( u^\tra \Sigma u ) = \sup_{u \in \Sp{d-1}} \Var ( u^\tra Y ) = \sup_{u \in \Sp{d-1}} \| \Sigma^{1/2} u \|^2 =   \| \Sigma^{1/2} \|^2_{\rm op} = \| \Sigma \|_{\rm op} . \end{equation}
Then $\sigma_\star$ is the largest eigenvalue of $\Sigma^{1/2}$; let $e_\Sigma \in \Sp{d-1}$ denote a unit-length eigenvector of $\Sigma^{1/2}$ corresponding to $\sigma_\star$.

\begin{example}[Diameter]
\label{ex:diameter-zero-drift}
Suppose that $d \in \N$ and~\eqref{ass:moments} holds with $\mu = \0$.
Then
\begin{equation}
\label{eq:diam-lil}
 \limsup_{n \to \infty}  \frac{\diam \cH_n}{\sqrt{2  n \log \log n}} =  \sigma_\star, \as  \end{equation}
To prove this, observe that, from~\eqref{eq:diameter-inequality} applied to $f \in U_d$ with $M = \Sigma^{1/2}$, 
\begin{equation}
\label{eq:diam-upper}
\sup_{f \in U_d} \diam ( \Sigma^{1/2} f [0,1] ) \leq \sup_{f \in U_d}  \|  \Sigma^{1/2}  \|_{\rm op} = \sigma_\star   , \end{equation}
by~\eqref{eq:eigenvalue}.
On the other hand,
\[ \diam ( \Sigma^{1/2} f [0,1] )  \geq \left| e_\Sigma^\tra ( \Sigma^{1/2} (f(1) - f(0) ) ) \right|
= \sigma_\star \left| e_\Sigma^{\tra} f(1)  \right| .\]
Take $f(t) = e_\Sigma t$, so $\|  f' (u) \|^2 =1$ and $f \in U_d$,
 to see that
\begin{equation}
\label{eq:diam-lower}
 \sup_{f \in U_d} \diam ( \Sigma^{1/2} f [0,1] ) \geq \sigma_\star . \end{equation}
Combining~\eqref{eq:diam-upper} and~\eqref{eq:diam-lower},
we obtain~\eqref{eq:diam-lil} by an application of Theorem~\ref{thm:khoshnevisan}
with $F (A) = \diam A$, 
noting that the diameter functional is homogeneous of order~$\beta =1$.  
This extends Khoshnevisan's Example~4.1~\cite[pp.~387--388]{khoshnevisan}, which had $\Sigma = I$.
\end{example}

\begin{remark}
\label{rem:diameter-zero-drift}
The result~\eqref{eq:diam-lil} may be compared with
\begin{equation}
\label{eq:norm-lil}
 \limsup_{n \to \infty}  \frac{\| S_n \|}{\sqrt{2 n \log \log n}} =  \sigma_\star, \as,  \end{equation}
as can be obtained either from Strassen's theorem (Theorem~\ref{thm:strassen-rw}), or directly from the one-dimensional Hartman--Wintner LIL~\cite[p.~382]{ct}. Indeed,
the one-dimensional random walk $u^\tra S_n$ satisfies
\[  \limsup_{n \to \infty}  \frac{u^\tra S_n}{\sqrt{2  n \log \log n}} =  \sigma_u, \as ,\]
where $\sigma^2_u := u^\tra \Sigma u$. This holds simultaneously for all~$u$ in a countable dense subset of~$\Sp{d-1}$,
which is enough to conclude~\eqref{eq:norm-lil}, since~$\| x \| = \sup_{u \in \Sp{d-1}} ( u^\tra x)$ and~$\sup_{u \in \Sp{d-1}} ( u^\tra \Sigma u ) = \sigma_\star^2$.
\end{remark}

\begin{example}[Volume]
\label{ex:volume-zero-drift}
Suppose that $d \in \N$ and~\eqref{ass:moments} holds with $\mu = \0$.
Then, there exists a constant $v_d \in (0,\infty)$ such that
\begin{equation}
\label{eq:volume-zero-drift-lil}
 \limsup_{n \to \infty}  \frac{V_d ( \cH_n )}{( 2  n \log \log n )^{d/2}} =  ( \det \Sigma )^{d/2} v_d, \as  , \end{equation}
where $v_d := \sup_{f \in U_d} V_d ( \hull f[0,1] )$.
Khoshnevisan~\cite[p.~389]{khoshnevisan} established~\eqref{eq:volume-zero-drift-lil} in the case $\Sigma = I$ and
proved that $v_2 = 1/(2\pi)$. The case of general~$\Sigma$ follows from Theorem~\ref{thm:khoshnevisan} with $F = V_d$ (a functional
that is homogeneous of order~$\beta=d$)
in the same way as described in~\cite{khoshnevisan}, accounting for the linear transformation $\Sigma^{1/2}$.
\end{example}

\section*{Acknowledgements}

We are grateful for the insightful and constructive comments received from an anonymous referee, which have led to improvements in the paper. 
Much early impetus for this work was provided by
stimulating discussions with 
James McRedmond and Vlad Vysotsky around 2019--20; in particular, James
identified the quadratic function~$\fopt$ as a likely candidate for the solution to the isoperimetric problem in Theorem~\ref{thm:planar-optimum},
and Vlad suggested a Strassen-type approach to shape theorems as in Corollary~\ref{cor:shape}.
We also thank Vlad for pointing out the link between our Theorem~\ref{thm:planar-optimum} and results in~\cite{av,vysotsky}.
The authors also express their thanks to the 
 Institute of Mathematical Stochastics at TU Dresden,
for hospitality in hosting a stimulating research visit in September 2022. This visit was supported by Alexander-von-Humboldt Foundation project No.~HRV 1151902 HFST-E. 
 The work of NS and S\v S was supported by Croatian Science Foundation grant 
 no.~2277.
The work of AW was supported by EPSRC grant EP/W00657X/1.

\end{document}